\documentclass[11pt]{amsart} 
\usepackage{geometry}             
\usepackage{float}
\usepackage{fullpage,xcolor}
\usepackage[mathscr]{euscript}
\usepackage[all]{xy}
\usepackage{epsfig}
\usepackage[T1]{fontenc}
\usepackage{amsfonts}

\usepackage{graphicx}
\usepackage{amssymb}
\usepackage{amsmath}
\usepackage{amsthm}
\usepackage{mathrsfs}
\usepackage{epstopdf}
\usepackage{url}
\usepackage[msc-links,alphabetic]{amsrefs}
\usepackage{upgreek}

\usepackage{color}
\usepackage{subfig} 

\newcommand{\be}{\beta}

\newcommand{\la}{\lambda}
\renewcommand{\phi}{\varphi}
\newcommand{\si}{\sigma}
\newcommand{\om}{\omega}

\newcommand{\Si}{\Sigma}
\newcommand{\ZZ}{{\mathbb Z}}

\newcommand{\RR}{{\mathbb R}}

\newcommand{\fD}{\mathfrak D}
\newcommand{\fK}{\mathfrak K}
\newcommand{\ind}{\operatorname{ind}}
\newcommand{\sign}{\operatorname{sign}}

\newcommand{\sm}{\smallsetminus}
\newcommand{\co}{\colon}

\newcommand{\wh}{\widehat}
\newcommand*\wbar[1]{%        % widebar 
  \hbox{ \kern-0.3em%
    \vbox{%
      \hrule height 0.5pt  % The actual bar
      \kern0.25ex%         % Distance between bar and symbol
      \hbox{%
        \kern-0.15em%       % Shortening on the left side
        \ensuremath{#1}%
        \kern-0.05em%       % Shortening on the right side
      }%
    }%
  \kern0.05em}%
}

\newtheorem{theorem}{Theorem} 
\newtheorem{lemma}[theorem]{Lemma}
\newtheorem{proposition}[theorem]{Proposition}

\newtheorem{corollary}[theorem]{Corollary}
\newtheorem{conjecture}[theorem]{Conjecture}

\theoremstyle{definition}     % this and the line below gives definitions in roman
\newtheorem{definition}[theorem]{Definition}

\theoremstyle{remark}

\newtheorem{example}[theorem]{Example}

\title[Virtual knot cobordism and bounding the slice genus]{Virtual knot cobordism and bounding the slice genus}
\author[H. U. Boden]{Hans U. Boden}
\address{Mathematics \& Statistics, McMaster University, Hamilton, Ontario}
\email{boden@mcmaster.ca}
\urladdr{math.mcmaster.ca/~boden}

\author[M. Chrisman]{Micah Chrisman}
\address{Mathematics, Monmouth University, West Long Branch, New Jersey}
\email{mchrisma@monmouth.edu}
\urladdr{https://micah46.wixsite.com/micahknots}

\author[R. Gaudreau]{Robin Gaudreau}
\address{Mathematics \& Statistics, McMaster University, Hamilton, Ontario}
\email{gaudreai@mcmaster.ca}


\date{\today}                                           % Activate to display a given date or no date

\begin{document}
\maketitle

\begin{abstract}
In this paper, we compute the slice genus for many low-crossing virtual knots.  For instance, we show that 1295 out of 92800 virtual knots with 6 or fewer crossings are  slice, and that all but 248 of the rest are not slice. Key to these results are computations of Turaev's graded genus, which we show extends to give an invariant of virtual knot concordance. The graded genus is remarkably effective as a slice obstruction, and we develop an algorithm that applies virtual unknotting operations to determine the slice genus of many virtual knots with 6 or fewer crossings.  

\end{abstract}

%%%%%%%%%%%%%%%%%%%%%%%%%%%%%%%%%%%%%%%%%%%%%%%%%%%%%%%%%%%%%%%%%%%%%%%%%%%%%%%%%%%%%%%%%%%%%%%%%%%%

\section*{Introduction}
Classical knot concordance is a subject that sits on the interface of 3- and 4-dimensional topology. Consequently, it exhibits some spectacular and peculiar phenomena reflective of the contrasting forces of smooth and topological manifolds in dimension four.  
There is a deep relationship between concordance of knots and homology cobordisms of 3-manifolds, and 
many of the breakthroughs in Khovanov and Floer homology theories have led to new concordance invariants of knots. These in turn have provided a more detailed and nuanced view of the intricate interplay between topological and smooth concordance for knots in dimension three. 

In \cite{Turaev-2004, Turaev-2008a}, Turaev developed the theory of knots in surfaces,  extending many notions, including concordance, to the new setting. Along the way, he introduced many new ideas and techniques. One key invariant  that has no counterpart in the classical theory is the graded genus $\vartheta(K)$ (see \cite{Turaev-2008a}).  We will show that it is invariant under concordance of virtual knots, and although it is rather difficult to compute, it turns out to be extremely effective as a slice obstruction.

Virtual knots were introduced  by Kauffman in his landmark paper \cite{Kauffman-1999}. He and many others have contributed to the rapid development of virtual knot theory. Those efforts have produced a plethora of newfangled invariants, along with new combinatorial notions of ``knottiness'' such as flat knots, free knots, knotoids, etc. The field has experienced a veritable explosion of research output. In his AMS Notices article, Sam Nelson dubbed it  ``the combinatorial revolution in knot theory'' \cite{Nelson-2011}. The advancement involved a creative mix of topological insight and algebraic/combinatorial ingenuity. As a byproduct, we now have more virtual knot invariants than ever before, and the challenge for geometric topologists is to figure out what sort of geometric information the many new invariants encode. Virtual knot concordance is an ideal setting for such an investigation.

Concordance of virtual knots was introduced in \cite{Carter-Kamada-Saito} by Carter, Kamada, and Saito, and it dovetails nicely with Turaev's theory  of concordance of knots in surfaces. Kauffman explored the notion of virtual knot cobordism in \cite{Kauffman-2015}, and Khovanov homology was generalized to virtual knots by  Manturov in \cite{Manturov-2007}. In \cite{Dye-Kaestner-Kauffman-2014}, they extended the Rasmussen invariant to virtual knots and used it to bound the slice genus (see also~\cite{Rushworth-2017}). Despite these successes, very few virtual knot invariants are known to be invariant under concordance, and our understanding of virtual knot concordance is still rather limited. Many fundamental questions remain unanswered, such as: (1) Which virtual knots are slice, and (2) Can one compute the virtual slice genus? In this paper, we develop practical and computational methods to address these questions for low-crossing virtual knots.

Answers to these questions for classical knots have been known for  some time. For instance at the time of writing, information on the 4-ball genus is available for nearly all classical knots with up to 12 crossings from the online resource \emph{Knotinfo} \cite{knotinfo}. 

The methods we develop here will produce bushel baskets of slice virtual knots with 4, 5 and 6 crossings and an even larger class of virtual knots with slice genus $g_s(K) =1.$ They also give effective bounds on the slice genus for the other low-crossing virtual knots. 

Our strategy is a blend of traditional and computational techniques. Specifically, we obtain lower bounds on $g_s(K)$ using Turaev's graded genus. We obtain upper bounds
by applying operations to the arrows of a Gauss diagram (see Figures \ref{GD-uo} and \ref{GD-crossed-saddle}) that give rise to genus one cobordisms from $K$ to the unknot. 
Because of the sheer number of calculations, we rely on computational methods to determine the graded genus and to execute searches over all possible unknotting operations. These are implemented in \emph{Mathematica} \cite{Mathematica}.

The following comparison between classical knots and virtual knots helps to  illustrate the challenge here.
Up to symmetry, there are 7 nontrivial classical knots with  six or fewer crossings. Exactly one of them is slice (the stevedore knot $6_1$), five others have 4-ball genus one, and one has 4-ball genus two (the torus knot $5_1$).  
In comparison, up to symmetry, there are 92800 nontrivial virtual knots with six or fewer crossings.  We will see that at least 1295 are slice, and another 50478 have slice genus one. 
The following tables summarize our findings.

\begin{table}[H]  
\begin{tabular}{|c|c|c|c|c|}
\hline
 Crossing & Virtual &   &   &      \\
number & knots & $\vartheta=0$ & slice & unknown  \\
\hline \hline
2 & 1 & 0 & 0 & 0\\ \hline
3 & 7 & 1 & 0 & 0\\ \hline
4 & 108 & 15 & 13 & 1\\ \hline
5 & 2448 & 59 & 45 & 11\\ \hline
6 & 90235 & 1476 & 1237 & 236\\ \hline
\end{tabular}
\bigskip
\caption{Virtual knots, graded genus, and sliceness.} \label{table-1}
\end{table}

The first table lists the total number of virtual knots, the number with graded genus zero, the number that are slice, and the number of unknown slice status. The results indicate just how often the graded genus $\vartheta$ correctly predicts whether a virtual knot is slice. The actual slicings of the virtual knots were performed by hand, and  for virtual knots up to 5 crossings they can be found at the end of the paper.
(All of the slice virtual knots we found are actually ribbon.)

The second table is the result of applying an unknotting algorithm to determine which virtual knots have slice genus $g_s(K) \leq 1.$ For negative knots, it also applies Theorem 6.8 of \cite{Dye-Kaestner-Kauffman-2014}, which computes the slice genus on the nose. 
Note that for positive or negative virtual knots with six crossings, Theorem \ref{g-s-bound} shows that they are in fact the only ones with slice genus equal to three.

\begin{table}[H]  
\begin{tabular}{|c|c|c|c|c|c|c|}
\hline
 Crossing & Virtual &   &   &   &   &    \\
 number & knots & $g_s=0$ & $g_s=1$ & $g_s=2$ & $g_s=3$ & unknown  \\
\hline \hline
2 & 1 & 0 & 1 & 0 & 0 & 0\\ \hline
3 & 7 & 0 & 7 & 0 & 0& 0\\ \hline
4 & 108 & 13 & 15 & 79 & 0 & 1 \\ \hline
5 & 2448 & 45 & 1805 & 206&  0 & 392\\ \hline
6 & 90235 & 1237 & 48650 & 4107 & 1751& 34490 \\ \hline
\end{tabular}
\bigskip
\caption{The slice genus of virtual knots.} \label{table-2}
\end{table}

We now give a brief overview of the contents of the rest of this paper.
In section \ref{sec-vkc}, we introduce cobordism and concordance of virtual knots and the virtual unknotting operations. We present a general technique for slicing and cobording virtual knots in terms of Gauss diagrams, and the main result is Theorem \ref{g-s-bound}, which provides an upper bound on the slice genus of a virtual knot in terms of its crossing number. 
In Section \ref{sec-concordance}, we  show that several invariants of virtual knots are invariant under virtual knot concordance, including the writhe polynomial $W_K(t)$ (see \S \ref{ssec-writhe}) and the graded genus $\vartheta(K)$ (see \S \ref{ssec-gg}). We also outline an algorithm for computing $\vartheta(K)$. In  Section \ref{sec-discussion}, we give a brief discussion of open problems suggested by our results. 
Table \ref{table-3} and Figure \ref{fig-slice} appears at the end of the paper; the first is a list of the graded genus and slice genus for virtual knots up to four crossings, and the second depicts slicings of all known slice virtual knots up to five crossings. Datasets of the results for virtual knots up to six crossings, including slicings of 1237 virtual knots with six crossings, can be found online at \cite{Boden-Chrisman-Gaudreau-2017t}.

In this paper, we work in the smooth category, and we use decimals (e.g. 2.1 and 4.98) to refer to specific virtual knots in Green's tabulation \cite{Green}.

%%%%%%%%%%%%%%%%%%%%%%%%%%%%%%%%%%%%%%%%%%%%%%%%%%%%%%%%%%%%%%%%%%%%%%%%%%%%%%%%%%%%%%%%%%%%%%%%%%%%

\section{Virtual knots, cobordism, and sliceness} \label{sec-vkc}
\subsection{Virtual knot cobordism} \label{cob-conc}
A virtual link diagram is a regular immersion of a disjoint union of circles into the plane. Each double point of the immersion is either a classical crossing or a virtual crossing. Classical crossings are drawn with one arc passing over the other, and virtual crossings are encircled. Two virtual link diagrams are equivalent if one can be transformed into the other by a sequence of planar isotopies, Reidemeister moves ($r1$)--($r3$), and virtual Reidemeister moves ($v1$)--($v4$). A virtual link is an equivalence class of diagrams under this relation.

\begin{figure}[ht]
\def\svgwidth{1.00\textwidth} 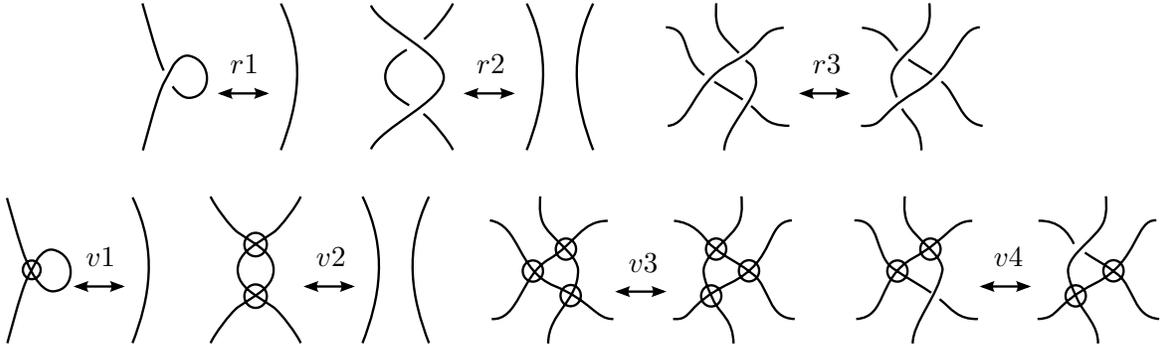
\caption{Generalized Reidemeister moves.}
\label{cob-movie}
\end{figure}

We will assume that all virtual knots and links here are oriented, indicated by placing an arrow on each component. We use $-K$ to denote the virtual knot or link with the opposite orientation.

Given two virtual knots $K_0$ and $K_1$, a cobordism between them is  a finite sequence of births, deaths, and saddles that transforms $K_0$ to $K_1$. The cobordism can be used to construct a connected surface co-bounding $K_0 \sqcup -K_1$, and the genus of this surface is $(s-b-d)/2$, where $s$ is the number of saddles, $b$ is the number of births, and $d$ is the number of deaths (see \cite{Chrisman-2016} for more details).

\begin{figure}[h]
\centering
 \includegraphics[scale=1.00]{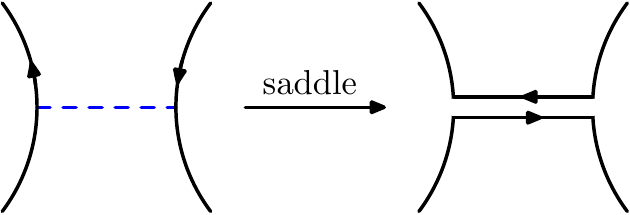} 
\end{figure} 

Virtual knots can also be represented as knots in thickened surfaces, and this gives a natural geometric context for studying cobordism and concordance of virtual knots. We take a moment to describe this approach, following \cite{Kamada-Kamada-2000} and \cite{Turaev-2008a}.

Given a virtual knot diagram, we can construct an oriented surface by thickening the arcs into bands, which are allowed to intersect at all real crossings and which overlap one another without intersecting at each virtual crossing. The resulting band surface has boundary a union of circles, and filling in its boundary components with 2-disks we obtain a closed oriented
surface $\Si$ together with a knot in $\Si \times I$, which is well-defined up to stable equivalence. (Stable equivalence is the addition or removal of one-handles to $\Si$ disjoint from the knot diagram.) 

In \cite{Carter-Kamada-Saito} they establish a one-to-one correspondence between virtual knots  
and stable equivalence classes of knots in thickened surfaces; they also  
develop corresponding notions of cobordism and concordance for knots in thickened surfaces  equivalent to those defined above in terms of virtual knot diagrams (see also \cite{Turaev-2008a}). 
Very briefly, 
two knots $\fK_0$ in $\Si_0 \times I$  and $\fK_1$ in $ \Si_1 \times I$ are said to be \emph{cobordant} if there exists a compact oriented $3$-manifold $M$ with  
$\partial M \cong -\Si_0 \sqcup \Si_1$ and an oriented 2-manifold $A$ embedded in $M \times I$ with $\partial A = -\fK_0 \sqcup \fK_1$. In case $A$ is an annulus, then $\fK_0$ and $\fK_1$ are said to be \emph{concordant}.

As with classical knots, there exists a cobordism from any virtual knot $K$ to the unknot \cite{Kauffman-2015}. One way to see this is to perform the ${\bf cd}$ move in Figures \ref{uo} and \ref{cd} to every classical crossing in  $K$. Conversely, every virtual knot can be obtained by a cobordism from the unknot, hence any two virtual knots are cobordant to one another.

\begin{definition}
The \emph{slice genus} of a virtual knot $K$, denoted $g_s(K)$, is the minimum genus over all cobordisms from $K$ to the unknot. 
 \end{definition}

Two virtual knots $K_0$ and $K_1$ are said to be \emph{concordant} if there is a genus zero cobordism between them, i.e. a sequence of $b$ births, $d$ deaths, and $s$ saddles connecting $K_0$ to $K_1$ such that 
\begin{equation}
s=b+d. 
\end{equation}
It is not true that every virtual knot is concordant to the unknot, and virtual knots that 
are concordant to the unknot are called \emph{slice}. 
A virtual knot $K$ is slice if and only if $g_s(K)=0.$

A \emph{ribbon concordance} from $K_0$ to $K_1$ is one with only saddles and deaths, and  virtual knots ribbon concordant to the unknot are called 
\emph{ribbon}. Just as with classical knots, every virtual knot that is ribbon is slice, and it is not known whether every virtual knot that is slice is in fact ribbon.

\subsection{Virtual unknotting operations}
The unknotting number of $K$ is a simple and yet elusive invariant of classical knots. Every classical knot $K$ can be changed to the unknot by a finite sequence of crossing changes, and the unknotting number of $K$ is defined as the minimum number of crossing changes needed to change $K$ to the unknot, taken over all knot diagrams. A standard argument shows that 
$u(K) \leq n/2$, where $u(K)$ is the
unknotting number of $K$ and $n$ is the number of crossings of $K$.  By another well-known argument, it is bounded below by the 4-ball genus, i.e. $g_4(K) \leq u(K)$, where  $g_4(K)$ denotes the 4-ball genus of $K$.  

Unlike classical knots, not all virtual knots can be unknotted by crossing changes. Indeed, crossing change does not alter the homotopy class of the virtual knot, and some virtual knots are homotopically nontrivial and therefore not unknottable using crossing changes. The theory of virtual knots up to crossing changes is called \emph{flat knot theory}, and flat knots can even be non-trivial up to concordance, as shown in \cite{Turaev-2004}. 

\begin{figure}[h]
\includegraphics[scale=0.75]{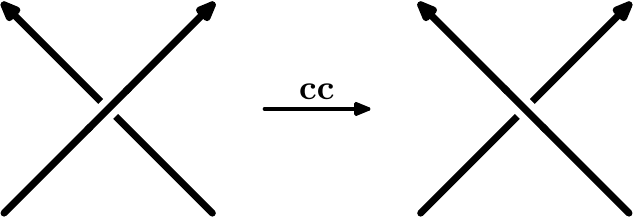} \qquad  \qquad 
\includegraphics[scale=0.75]{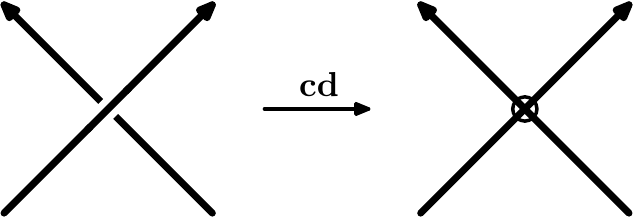}

\bigskip
\includegraphics[scale=0.75]{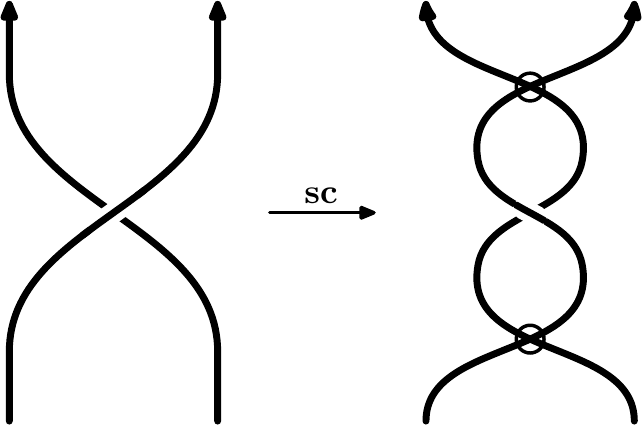}  \qquad  \qquad 
\includegraphics[scale=0.75]{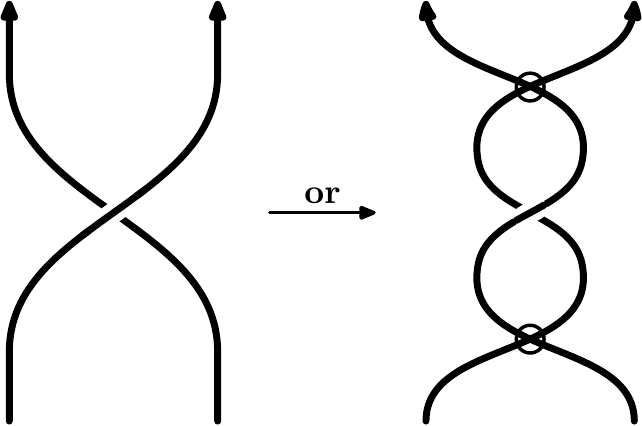} 
\caption{Four unknotting operations: ${\bf cc} = $ crossing change, ${\bf cd} = $ chord deletion, ${\bf sc} = $ sign change, and ${\bf or} = $ orientation reversal. }
\label{uo}
\end{figure}

Using the four operations in Figure \ref{uo}, one can unknot any virtual knot. In fact, the chord deletion alone is sufficient to unknot any virtual knot. To see this, apply {\bf cd} to all real crossings, and notice that any virtual knot diagram with only virtual crossings is trivial. One way to see that is to directly apply virtual Reidemeister moves to the virtual knot diagram to remove all the virtual crossings, another way is to represent the given virtual knot in terms of its Gauss diagram, which will have no chords and is therefore trivial.

The inclusion of the other three moves allows one to more quickly simplify a given virtual knot, and we define the \emph{virtual unknotting number} of a virtual knot $K$, denoted $u_v(K)$, to be the minimum number of applications of $\{{\bf cc}, {\bf cd}, {\bf sc}, {\bf or}\}$ that are needed to convert it to the unknot, with the minimum taken over all diagrams representing $K$. The preceding argument shows that $u_v(K) \leq n,$ where $n$ is the crossing number of $K$. One can improve the inequality slightly to  $u_v(K) \leq n-1$, since any virtual knot diagram with only one crossing represents the unknot.  

One can also define the \emph{virtual slicing number} for a virtual knot $K$ to be the minimum number of applications of $\{{\bf cc}, {\bf cd}, {\bf sc}, {\bf or}\}$ that are needed to convert it to a slice knot. Denoting this number by $s_v(K)$, we note that it is bounded above by the virtual unknotting number and below by the slice genus, i.e.  
$$g_s(K) \leq s_v(K) \leq u_v(K).$$ 
The second inequality is immediate and the first will be explained in the next subsection.

\subsection{Gauss diagrams}
A Gauss diagram is a decorated trivalent graph consisting of one or more core circles, oriented counterclockwise, together with a finite collection of directed chords decorated with signs ($\pm$) connecting distinct pairs of points on the circles. Each core circle represents a component of the virtual link diagram, and directed chords connect the pre-images of double points corresponding to the classical crossings. Chords point from the over-crossing arc to the under-crossing arc and their sign indicates whether the crossing is right-handed ($+$) or left-handed ($-$). Thus associated to every virtual link diagram is a Gauss diagram, and given a Gauss diagram, one can draw a virtual link diagram realizing it. 
The generalized Reidemeister moves can be translated into moves between Gauss diagrams (cf. Figure \ref{GD-RM2}), and in this way we see that a virtual knot or link is an equivalence class of Gauss diagrams under the resulting equivalence. 

\begin{figure}[H]
\includegraphics[scale=0.80]{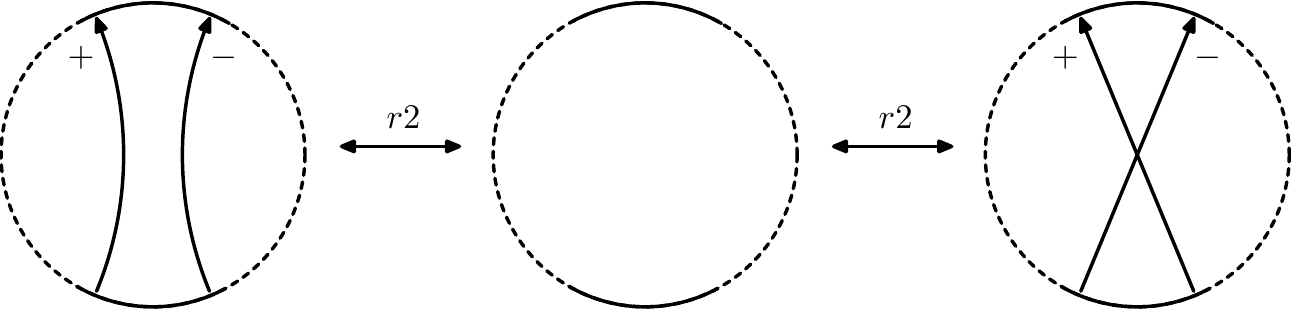} 
\caption{Reidemeister 2 moves on Gauss diagrams.}
\label{GD-RM2}
\end{figure}

\begin{figure}[h]
\includegraphics[scale=0.770]{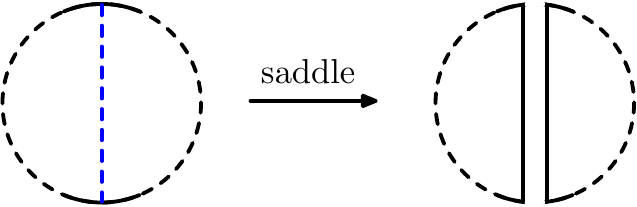} \qquad  
\includegraphics[scale=0.770]{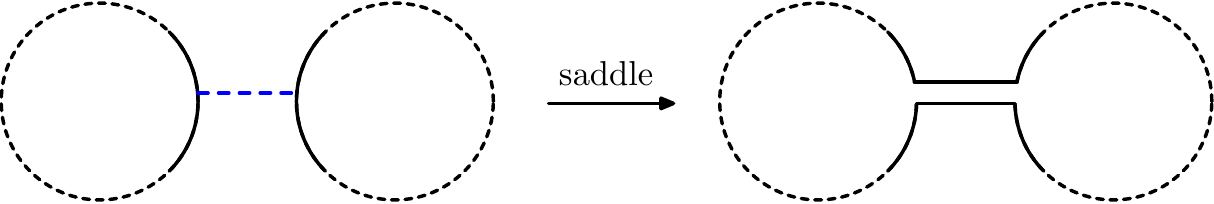}  
\caption{A fission saddle (left) and a fusion saddle (right) on Gauss diagrams.}
\label{GD-saddle}
\end{figure}

Cobordism and concordance of virtual knots can be performed on the Gauss diagrams, and it is a convenient alternative to working on virtual knot diagrams. Figure \ref{GD-saddle} shows a fission and fusion saddles, and  
Figure \ref{cob-movie} gives a cobordism movie that shows the virtual knot 3.1 has slice genus one.
Above are the moves on the virtual knot diagrams, and below them are the corresponding moves on the Gauss diagram.
 
\begin{figure}[ht]
\def\svgwidth{0.95\textwidth} 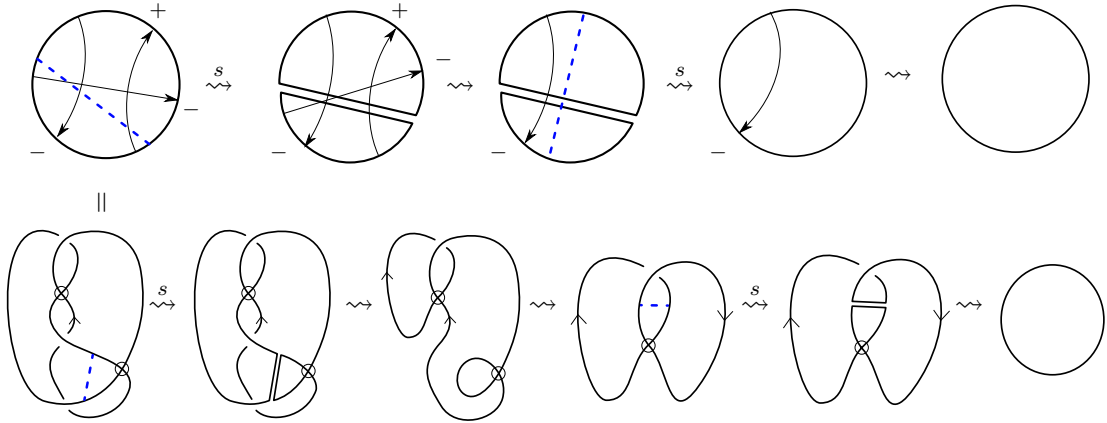
\caption{A movie of a genus one cobordism of the virtual knot 3.1.}
\label{cob-movie}
\end{figure}

Using the same technique, one can slice virtual knots by applying saddles, births, and deaths to their Gauss diagrams. Figure \ref{4-71} illustrates the saddle move and shows the virtual knot 4.71 is slice. The saddle is depicted on the virtual knot diagram and its Gauss diagram, and it is not difficult to see from either approach that the resulting two component link is trivial under Reidemeister moves. Slicings for all other known slice virtual knots with crossing number 4 and 5 are given in Table \ref{fig-slice} at the end of this paper. For slicings of six crossing knots, see \cite{Boden-Chrisman-Gaudreau-2017t}.

\begin{figure}[ht]
\includegraphics[scale=0.90]{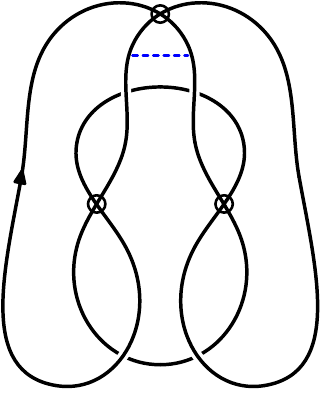}  \qquad  \qquad 
\includegraphics[scale=1.00]{Figures/GD-4-71} 
\caption{The virtual knot 4.71 is slice.}
\label{4-71}
\end{figure}

The four unknotting operations can also be viewed as moves on Gauss diagrams, see Figure \ref{GD-uo}, and each one can be realized by a genus one cobordism. This is in fact well-known for the crossing change in classical knot theory, and it is the reason for the relationship the 4-ball genus $g_4(K)$ of a knot and its unknotting number $u(K)$ mentioned earlier. For the other three virtual unknotting operations, one can realize them as genus one cobordisms by applying a fission saddle followed by a fusion saddle.

\begin{figure}[h]
\includegraphics[scale=0.80]{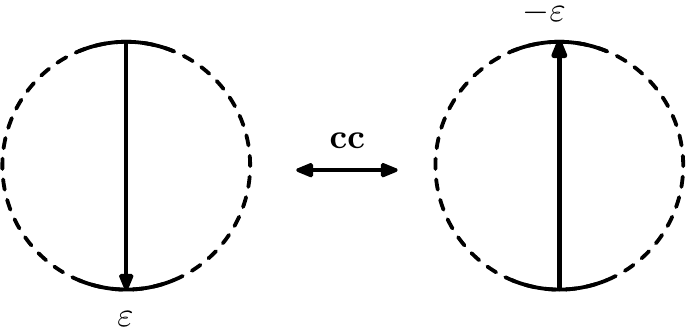} \qquad \qquad  
\includegraphics[scale=0.80]{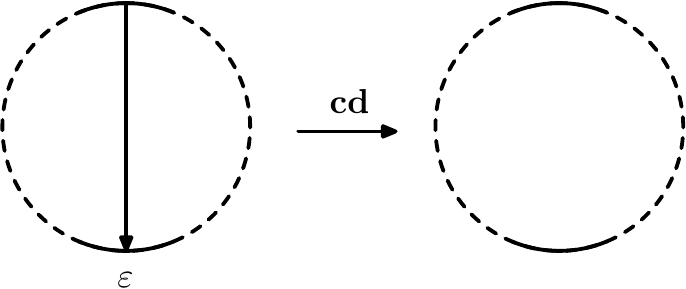}

\bigskip
\includegraphics[scale=0.80]{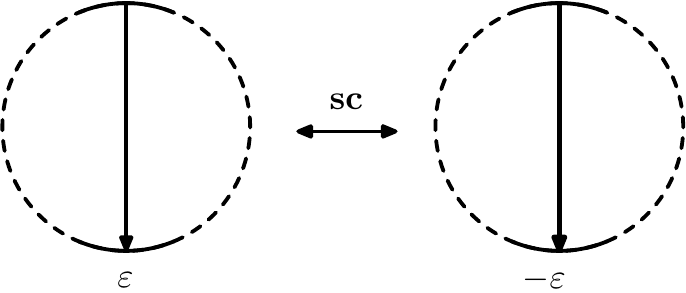}  \qquad  \qquad 
\includegraphics[scale=0.80]{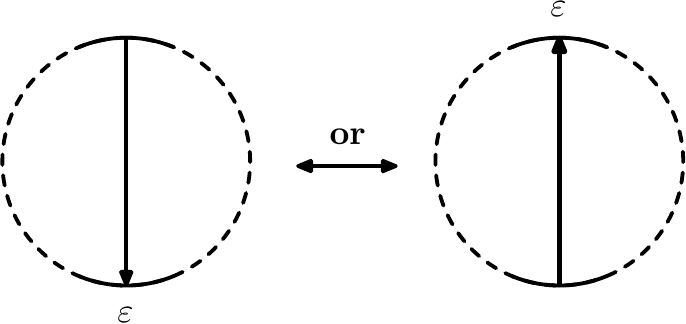} 
\caption{The four unknotting operations on Gauss diagrams.}
\label{GD-uo}
\end{figure}

For instance, consider the chord deletion operation ${\bf cd}$. Below is a movie that realizes ${\bf cd}$ as a genus one cobordism. 
The top row shows the moves on the Gauss diagram,
and below it are the corresponding moves on the virtual knot diagram. 
Similar arguments show that the other unknotting operations can be realized as genus one cobordisms.
In the other cases, use Reidemeister moves to insert the appropriate chord right before the second saddle.   

\bigskip 
\begin{figure}[H]
\def\svgwidth{0.95\textwidth} 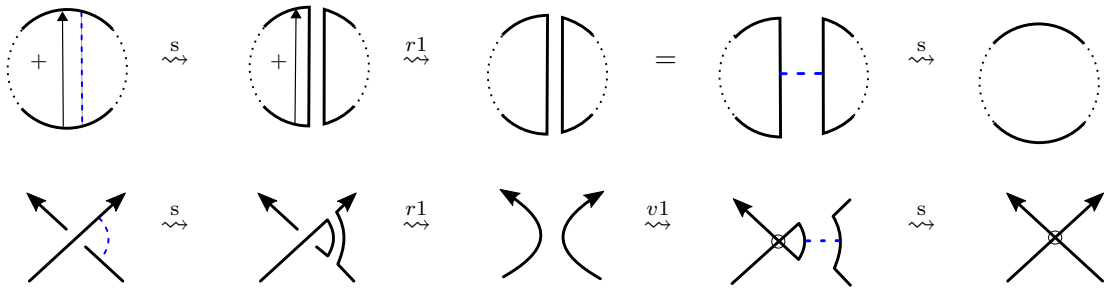
\caption{Chord deletion realized as a genus one cobordism.}
\label{cd}
\end{figure}

In general, two saddle moves applied to a Gauss diagram will result in a genus one cobordism provided that the first is a fission saddle and the second is a fusion saddle. This occurs when the saddles are performed along intersecting chords, hence, we call this operation the \emph{crossed saddle move}. It can be used to cancel a pair of intersecting chords in a genus one cobordism. The idea is to place the two saddles parallel to the two chords and  use $r1$ moves to eliminate each of the chords after performing the saddles. In the next subsection, we will use the crossed saddle move to obtain upper bounds on the slice genus.
\begin{figure}[h]
\includegraphics[scale=0.85]{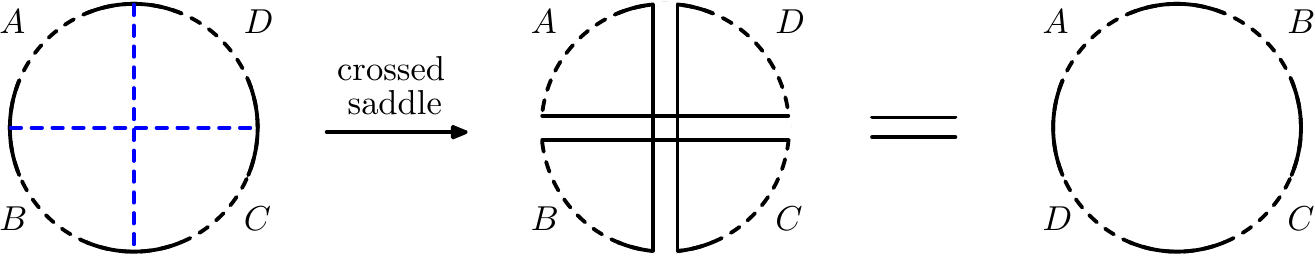}
\caption{The crossed saddle move.}
\label{GD-crossed-saddle}
\end{figure}

\subsection{Upper bounds on the slice genus}
In this subsection, we explain how to use the crossed saddle move to obtain upper bounds on the slice genus of a virtual knot $K$ in terms of its crossing number. 

\begin{theorem} \label{g-s-bound} 
If $K$ is a virtual knot with crossing number n,
then $g_s(K) \leq n/2.$ If $K$ contains both positive and negative crossings, then $g_s(K) \leq (n-1)/2$.
\end{theorem}

Before proceeding with the proof, we explore some of the immediate applications. For instance, any virtual knot $K$ with three or fewer crossings must have slice genus $g_s(K) \leq 1,$ and any virtual knot $K$ with five or fewer crossings must have slice genus $g_s(K) \leq 2.$ Combined with computations of the graded genus listed in Table \ref{table-3}, this shows that most virtual knots with four or fewer crossings have slice genus one.
 
\begin{corollary} 
Suppose $K$ is a virtual knot with $n=2m$ crossings and maximal slice genus $g_s(K)=m$. If $D$ is a Gauss diagram  representing $K$ with $n$ chords, then the chords of $D$ are either all positive or all negative.
\end{corollary}

This corollary is particularly useful when combined with the following result of Dye, Kaestner, and Kauffman \cite[Theorem 6.8]{Dye-Kaestner-Kauffman-2014}, which they proved using the virtual Rasmussen invariant arising from Khovanov homology for virtual knots.
Note that it is stated for positive virtual knots, but it holds equally well for negative virtual knots.

\begin{theorem}[Dye, Kaestner, Kauffman] \label{thm-DKK}
If $K$ is a virtual knot with all positive crossings, then $g_s(K) = (n-r+1)/2$, where $n$ is the number of real crossings and $r$ is the number of virtual Seifert circles.
\end{theorem}

\noindent
\emph{Proof of Theorem \ref{g-s-bound}.}
First, we argue by induction that for any Gauss diagram with $n$ chords, its slice genus satisfies $g_s(K) \leq n/2.$ 
The statement is obviously true for Gauss diagrams with 1 or 2 chords. In the first case, the virtual knot is trivial, and in the second, it is equivalent to 2.1 or its mirror image, which one can easily verify has slice genus one. 

Suppose then that $D$ is a Gauss diagram with $n>2$ chords and represents the virtual knot $K$. If $D$ is nontrivial, then $D$ must contain a pair of intersecting chords, and performing the crossed saddle move gives a genus one cobordism from $D$ to a new Gauss diagram $D'$ with $n - 2$ chords. 
Let $K'$ be the virtual knot represented by $D'$. The induction hypothesis gives that $g_s(K') \leq (n-2)/2$. Since there is a genus one cobordism from $K$ to $K'$, it follows that $g_s(K') \leq n/2$.

This completes the proof of the first inequality, and since the second holds automatically when $n$ is odd, we assume $n$ is even and that $D$ contains both positive and negative chords.
The proof is by induction on $n$, and since any Gauss diagram $D$ with only two chords and of opposite signs represents the unknot, we can further assume that $n \geq 4$.

Let $D$ be a Gauss diagram with $k$ positive chords and $\ell$ negative chords, where $k,\ell >0$ satisfy $k+\ell=n$. Since the slice genus of a virtual knot is unchanged under mirror images, we can further assume without loss of generality that $k \leq \ell.$

\smallskip \noindent
{\bf Case 1:} The Gauss diagram $D$ contains a pair of intersecting chords of opposite sign and $k>1$. 
Using them, we can produce a genus one cobordism to a Gauss diagram $D'$. Since $k,\ell>1$, the new Gauss diagram $D'$ will contain both positive and negative chords and have $n'=n-2.$

\smallskip \noindent
{\bf Case 2:} The Gauss diagram $D$ contains a pair of intersecting negative chords and $\ell >2.$ Using them, we can produce a genus one cobordism to a Gauss diagram $D'$. Since $\ell > 2$, the new Gauss diagram $D'$ will contain both positive and negative chords and have $n'=n-2.$

\smallskip \noindent
{\bf Case 3:} The Gauss diagram does not contain a pair of intersecting chords of opposite sign, and $k=\ell=2.$ Then $n=k+\ell=4$, and the virtual knot $K$ is a connected sum of $2.1$ and its mirror image. There are several possibilities, and one can show directly that in each case $K$ has slice genus $g_s(K)\leq 1.$

\smallskip \noindent
{\bf Case 4:} The Gauss diagram $D$ does not contain a pair of intersecting negative chords and $k=1$.
If $D$ contains a negative chord that does not intersect the positive chord, then that chord can be eliminated by a $r1$ move, and so the virtual knot $K$ can be represented by a Gauss diagram with $n-1$ chords, and we have that $g_s(K) \leq (n-1)/2$ by the first inequality. Otherwise, the negative chords of $D$ must be parallel to each other and intersect the positive chord. Performing a chord deletion  along the positive chord, we obtain a genus one cobordism to the unknot, showing that $g_s(K) \leq 1\leq n-1$ as claimed. This completes the argument. 
\qed

%%%%%%%%%%%%%%%%%%%%%%%%%%%%%%%%%%%%%%%%%%%%%%%%%%%%%%%%%%%%%%%%%%%%%%%%%%%%%%%%%%%%%%%%%%%%%%%%%%%

\section{Concordance Invariants} \label{sec-concordance}

In this section we show that several invariants of virtual knots are also concordance invariants. First we give a topological proof that several index invariants are concordance invariants. Secondly, we show how Turaev's graded genus, which is a concordance invariant for knots in thickened surfaces, is also a concordance invariant of virtual knots. The polynomial invariants are useful as slice obstructions because they are easily computable. The graded genus is more difficult to compute, but it provides useful bounds on the slice genus. For instance, we apply the graded genus to show that every non-classical virtual knot with only positive crossings is not slice (Theorem \ref{gg-app}).

\subsection{Index polynomials} \label{ssec-writhe} 
Here we will consider just three variants: the odd writhe (see \cite{Kauffman-2004}), the Henrich-Turaev polynomial (see \cite{Henrich-2010, Turaev-2008a}), and the writhe polynomial (see \cite{Cheng-Gao}); many other variants can be found in  the literature. In \cite{Chrisman-Kaestner}, the HT-polynomial was shown to be a concordance invariant of virtual knots. In \cite{Rushworth-2017}, the odd writhe was shown to be a concordance invariant using virtual Khovanov homology. We will give a topological argument to show that the writhe polynomial is a concordance invariant. Since the writhe polynomial dominates the other two invariants, our result immediately implies that the odd writhe and Henrich-Turaev polynomial are also concordance invariants. We begin by reviewing the necessary definitions. 

\begin{figure}[htb]
\begin{tabular}{c} \\
\def\svgwidth{1.25in}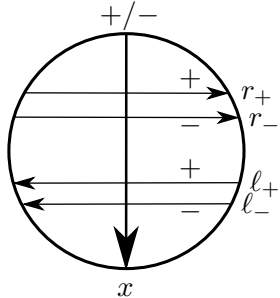 \end{tabular}
\caption{Definition of $\ind(x)$.} \label{fig_cheng_defn}
\end{figure}

Let $K$ be a virtual knot with Gauss diagram $D$, and let $C(D)$ the set of chords of $D$. For every $x\in C(D)$, we define the \emph{index of} $x$ as follows. Rotate $D$ so that the arrow $x$ points downwards. Let $r_+$ ($\ell_+$) denote the number of positive arrows crossing $x$ from left to right (respectively, from right to left). Similarly define $r_-$, $\ell_-$ for negative arrows crossing $x$. See Figure \ref{fig_cheng_defn}. Then the index of $x$ is given by:
\[
\ind(x)=r_+-r_- - \ell_+ +\ell_-.
\]
Denote by $\sign(x)$ the sign of crossing $x$, so that $\sign(x)=\pm 1$. Then the odd writhe $J(K)$, Henrich-Turaev polynomial $P_K(t)$, and writhe polynomial $ W_K(t)$  are defined by setting  
\begin{eqnarray} \label{writhe-1}
J(K) &=&\sum_{\ind(x) \text{ odd}} \text{sign}(x),\\ W_K(t) &=& \sum_{\ind(x) \ne 0} \sign(x) \, t^{\ind(x)}, \, \, \text{and} \\ P_K(t)&=&\sum_{\ind(x) \ne 0} \sign(x) \, t^{|\ind(x)|}.
\end{eqnarray}
The sum in each of the above is taken over all $x \in C(D)$ satisfying the stated condition.\footnote{Note that a slightly different normalization is used here for $W_K(t)$ than in \cite{Cheng-Gao}.} 
\begin{example} 
Figure \ref{fig_wk_example} shows the Gauss diagram of a virtual knot $K$ with $W_K(t)=t^2-2t+2t^{-1}-t^{-2}$, $P_K(t)=0$, and $J(K)=0$. This shows that the writhe polynomial is a stronger concordance invariant than both the odd writhe and the Heinrich-Turaev polynomial.
\end{example} 

\begin{figure}[htb]
\def\svgwidth{1.25in}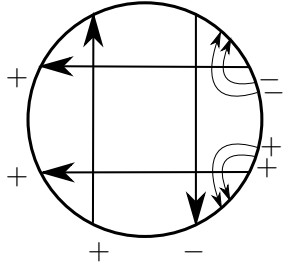
\caption{A virtual knot with trivial $P_K(t)$ but non-trivial $W_K(t)$.} \label{fig_wk_example}
\end{figure}

The index of a crossing can be described topologically using knots in thickened surfaces. Suppose $\fK$ is a knot in $\Si \times I$ and is represented by a diagram $\fD$ on $\Si$. Let $C(\fD)$ denote the set of crossings of $\fD$. If $x \in C(\fD)$, let $\fD_x$ denote the \emph{right half} (or \emph{distinguished half}) of the diagram obtained by performing the oriented smoothing of $\fD$ at $x$ (see Figure \ref{fig_right_half}); it is the component on the right of the crossing when both arcs are oriented upwards. We use $\fD_x \cdot \fD$ to denote the algebraic intersection number of the homology classes
$[\fD_x], [\fD] \in H_1(\Si)$.

\begin{figure}[htb]
\begin{tabular}{c} \\
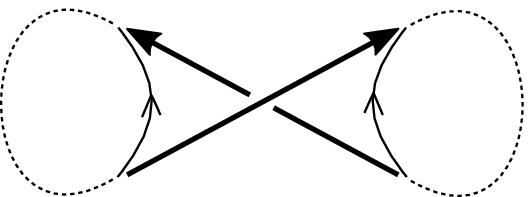
\qquad 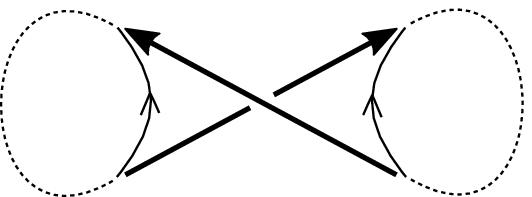
\end{tabular}
\caption{Definition of right half of a crossing.} \label{fig_right_half}
\end{figure}

Then the writhe polynomial and Henrich-Turaev polynomial are defined for knots in thickened surface by setting:
\begin{equation} \label{writhe-2}
W_{\fK}(t)=\sum_{\fD_x \cdot \fD \ne 0} \sign(x) \, t^{\sign(x) (\fD_x \cdot \fD)} \quad \text{ and } \quad P_{\fK}(t)=\sum_{\fD_x \cdot \fD \ne 0} \sign(x) \, t^{|\fD_x \cdot \fD|}.
\end{equation} 

\begin{figure}[htb]
\begin{tabular}{cccc} \\
\def\svgwidth{1.5in}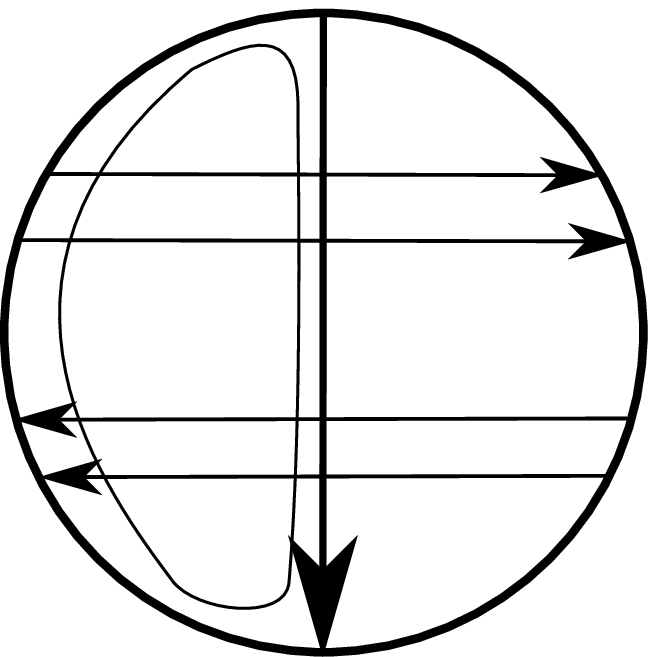 & & & \def\svgwidth{1.5in}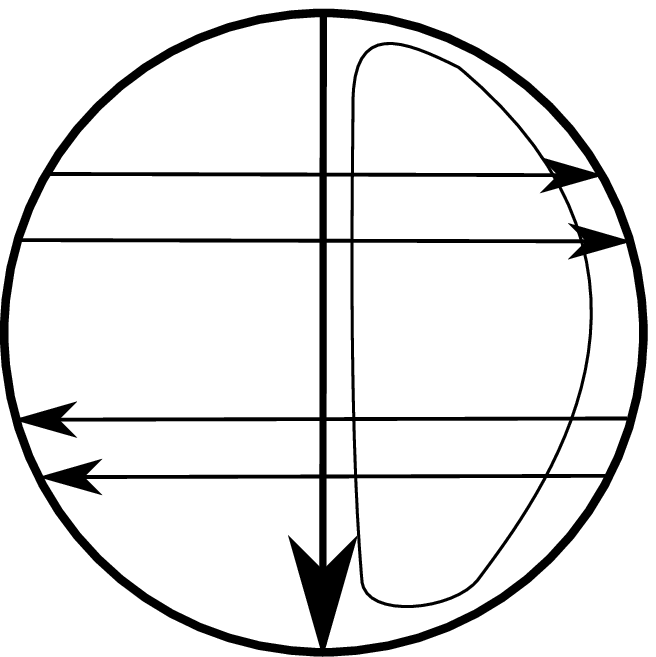\end{tabular}
\caption{Relating $\ind(x)$ and $\fD_x \cdot \fD$.} \label{fig_cheng_is_mod}
\end{figure}

\begin{lemma} \label{thm_W_eq_C} 
Suppose $K$ is a virtual knot, $\fK$ a knot in $\Si \times I$ representing $K$, and $\fD$ is a diagram for $\fK$. For every crossing $x$ of $\fD,$ we have
$\ind(x)=\sign(x) \, (\fD_x \cdot \fD)$.
\end{lemma}
\begin{proof} If $x$ is a positive crossing, then  the distinguished half $\fD_x$ is as depicted  in Figure \ref{fig_right_half} (left), in particular it lies to the left of the arrow of $x$. 
If $x$ is a negative crossing, then the  distinguished half $\fD_x$ lies to the right of the arrow (see Figure \ref{fig_right_half} (right)). In either case $\fD_x \cdot \fD$ is the signed sum of local intersections of $\fD_x$ with $\fD$. When $\sign(x)=1$, this signed sum coincides with $r_+-r_--(\ell_+-\ell_-)=\ind(x)$.  
When $\sign(x)=-1$, it equals $\ell_+-\ell_--(r_+ -r_-)=-\ind(x)$. Thus $\ind(x) = \sign(x) (\fD_x \cdot \fD)$ as claimed, and that completes the proof of the lemma.
\end{proof}

Comparing equations \eqref{writhe-1} and \eqref{writhe-2}, Lemma \ref{thm_W_eq_C} shows that $W_K(t)= W_\fK(t)$ and $P_K(t)= P_\fK(t)$. Similarly, the odd writhe of a virtual knot can be computed from any representative in a thickened surface.

\begin{theorem} \label{thm_writhe_poly} 
$J(K)$, $W_K(t)$, and $P_K(t)$ are concordance invariant of virtual knots.
\end{theorem}

The following lemma is helpful in proving  Theorem \ref{thm_writhe_poly} (cf.~\cite[Lemma 2.3.2]{Turaev-2008a}).

\begin{lemma}\label{lemma_turaev_invol} Suppose $\fK$ is a knot in $\Si \times I$ with diagram $\fD$. If $\fK$ is slice, then there is an involution $\nu$ on the set $C(\fD)$ such that, for all $x,y \in C(\fD)$:
\begin{enumerate}
\item[(i)] If $\nu(x)=x$, then $\fD_x\cdot\fD=0$.
\item[(ii)]  If $\nu(x)=y$ and $x\neq y$, then $\sign(x)=-\sign(y)$ and $\fD_x \cdot \fD = -\fD_y \cdot \fD$.
\end{enumerate}
\end{lemma}
\begin{proof} The proof is sketched;  it is similar to the proof of \cite[Lemma 2.3.2]{Turaev-2008a}. Since $\fK$ is slice, we have a 3-manifold $M$ with $\partial M = \Si$ and 2-disk $D \subset M \times I$ with $\partial M = \fK.$ Let $f\colon D \hookrightarrow M \times I \to M$ be the composition of inclusion and projection, and we may assume that $f$ is a generic map, namely that it is an immersion outside a finite number of branch points. The set of points $x\in D$ satisfying $|f^{-1}(f(x))|\ge 2$ consists of embedded circles and intervals in $D$ whose intersections correspond to the triple-points of $f$. The endpoints of the immersed intervals lie either in the set $C(\fD)$ or in the set of branched points. If $x \in C(\fD)$ connects to a branch point, define $\nu(x)=x$. Otherwise, $x$ connects to exactly one other point $y \in C(\fD)$ such that $y \ne x$. In that case, we set $\nu(x)=y$. The claim then follows from a detailed analysis of the right-halves $\fD_x$ and $\fD_y$ in each case.
\end{proof}

\begin{proof}[Proof of Theorem \ref{thm_writhe_poly}] Consider the case of $W_K(t)$. Since the writhe polynomial is additive under connected sum and satisfies $W_{-K}(t)=-W_{K}(t)$ (cf.~\cite[Proposition 3.4]{Cheng-Gao}), it is enough to show that $W_{K}(t)=0$ whenever $K$ is a slice knot. By Lemma \ref{thm_W_eq_C} and the equivalence of the two notions of concordance, it is sufficient to prove the corresponding statement for knots in thickened surfaces. Suppose then that $\fK$ is a slice knot in $\Si \times I$. Then there is an oriented $3$-manifold $M$ with $\partial M = \Si$ and a 2-disk $D$ in $M \times I$ such that $\partial D=\fK$. We apply Lemma \ref{lemma_turaev_invol} and analyze the orbits of the involution $\nu$.

Part (i) of Lemma \ref{lemma_turaev_invol} implies that crossings $x$ fixed by $\nu$ do not contribute to $W_{\fK}(t)$. Suppose then that $x,y$ are distinct crossings with $\nu(x)=y$. Part (ii) of Lemma \ref{lemma_turaev_invol} then shows that
\[
\sign(x)(\fD_x\cdot\fD)=(-\sign(y))(-\fD_y\cdot\fD)=\sign(y)(\fD_y\cdot\fD)
\]
Since $\sign(x)=-\sign(y)$ and their contributions to $W_{\fK}(t)$ have the same exponent, it follows that their total contribution to $W_{\fK}(t)$ is zero. This implies that $W_{\fK}(t) =0$. The proofs for $P_K(t)$ and $J(K)$ follow similarly. This completes the proof. 
\end{proof}

\subsection{Turaev's Graded Genus} \label{ssec-gg}
For the purposes of determining sliceness and computing the slice genus of virtual knots, the most powerful invariant is Turaev's graded genus of knots in thickened surfaces.  Here we provide a practical guide to Turaev's graded genus concordance invariant. There are three steps needed to find the graded genus of a knot in $\Si \times I$ efficiently: calculating the \emph{graded matrix}, determining the \emph{simple graded fillings}, and lastly computing the graded genus itself.

\subsection*{Graded Matrices} Let $\fK$ be a knot in $\Si \times I$ represented by a diagram $\fD$ on $\Si$. Let $C=C(\fD)$ be the set of crossings of $\fD$, and set $C_+$ to be the subset of $C$ of positive crossings and $C_-$ the negative crossings. Consider the augmented set $\wh{C} = C \cup \{s\}$. If $x \in C$, let $\fD_x$ be as before the \emph{right half} of the knot diagram of the oriented smoothing of $fD$ at $x$ (see Figure \ref{fig_right_half}). For the special element $s \in \wh{C}$, set 
$\fD_s = \fD.$
Define $\be \co \wh{C} \times \wh{C} \to \ZZ$ by setting
$$\be(x,y)= [\fD_x] \cdot [\fD_y],$$
where $[\fD_x], [\fD_y] \in H_1(\Si)$ denote homology classes and $\cdot$ is the intersection pairing on $\Si$.

The \emph{graded matrix} of $\fD$ is the triple $T=(\wh{C},s,\be)$. Note that $\be$ is skew-symmetric. The map $\be$ can be computed directly from a Gauss diagram of $\fD$. This will be explained in the next subsection.

\subsection*{Simple Graded Fillings} Let $\ZZ[\wh{C}]$ be the free $\ZZ$-module generated by the set $\wh{C}$. A \emph{simple graded filling} is a finite subset $\la$ of $\ZZ[\wh{C}]$ satisfying the following three properties:
\begin{enumerate}
\item $s \in \la$,
\item $\displaystyle{\sum_{x \in \la} x=\sum_{y \in \wh{C}} y}$, and
\item if $x \in \la$ and $x \not\in \wh{C}$, then $x=g+h$ for $g,h \in C$ of opposite sign.
\end{enumerate} 

\begin{example} Suppose that $C_+=\{g_1,g_2\}$ and $C_-=\{g_3\}$. Then the only simple graded fillings are $\{s,g_1,g_2,g_3\}$, $\{s,g_1+g_3,g_2\}$, $\{s,g_1,g_2+g_3\}$.
\end{example}

Next extend $\be$ to a skew-symmetric bilinear form on $\ZZ[\wh{C}]$ by linearity. The \emph{matrix of a simple graded filling} $\la=\{\la_1,\ldots,\la_m\}$ is the $m\times m$ matrix $B_{\la}$ whose $(i,j)$-entry is $\be(\la_i,\la_j)$. The set $\wh{C}$ is itself a simple filling. The matrix of $\wh{C}$ is just the matrix associated to the form $\be$.
 
\subsection*{Computing $\vartheta(T)$} For a simple graded filling $\la$, define $\vartheta(\la)=\frac{1}{2} \cdot \text{rank}_{\ZZ}(B_{\la})$. The \emph{graded genus} $\vartheta(T)$ is then defined to be:
\[
\vartheta(T)=\min\{\vartheta(\la): \la \text{ is a simple graded filling}\}.
\]
Given a sequence of graded matrices $T_1,\ldots, T_p$, one may analogously define $\vartheta(T_1,\ldots, T_p)$. First define a form by $\be=\bigoplus_{i=1}^p \be_i \co \ZZ[\wh{C}] \times \ZZ[\wh{C}] \to \ZZ$ where $\wh{C}=\bigsqcup_{i=1}^p \wh{C}_i$ and ${C}=\bigsqcup_{i=1}^p {C}_i$. (Thus $C = \wh{C} \sm \{s_1 \ldots, s_p\}.$) Similarly define $C_+$ ($C_-$) to be the disjoint union of all the $+$ (respectively, $-$) crossings. Next simple grading fillings are generalized to \emph{graded fillings}. Let $S$ be the submodule of $\ZZ[\wh{C}]$ generated by $s_1,\ldots, s_p$. Then a finite subset of $\ZZ[\wh{C}]$ is called a graded filling if it satisfies all of the following:
\begin{enumerate}
\item $s_1+\cdots+s_p \in \la$,
\item $\displaystyle{\sum_{x \in \la} x \equiv \sum_{g \in \wh{C}} g \pmod{S}}$, and
\item if $x \in \la$, then $x \in S$, or $x \equiv g \pmod{S}$ for some $g \in C$, or $x \equiv g+h \pmod{S}$ for some $g,h \in C$ of opposite crossing sign.
\end{enumerate} 
As before $\vartheta(T_1,\ldots,T_p)$ is the minimum $\vartheta(\la)$ over all graded fillings $\la$.

Given a graded matrix $T=(\wh{C},s,\be)$, let $-T=(-\wh{C},s,-\be)$, where $-\wh{C}$ is determined by $-C_+:=C_-$ and $-C_-:=C_+$.

\subsection{Graded Genus of a Virtual Knot Diagram} 
In this subsection, the graded genus is defined  for virtual knots in terms of their Gauss diagrams. We adopt parallel notation $(\wh{C},s,\be)$ as above for the graded matrix of virtual knots.

Let $D$ be the Gauss diagram of a virtual knot having $n$ arrows. Let $\wh{C} = C \cup \{s\}$ be the set of $n+1$ elements consisting of the arrows of $D$ together with an element $s$. Let $C_+,C_- \subseteq C$ be the subsets of positive and negative arrows of $D$. Assign a base point $*$ to $D$; its position does not affect the graded genus. Moving counterclockwise, number the arrows of $D$ in the same way as one determines the Gauss code.  The matrix of the map $\be$ will be an $(n+1) \times (n+1)$ matrix indexed by $0, 1, \ldots, n$. The $0$ index corresponds to $s$ while the indices $1,\ldots,n$ correspond to the arrow numbers.

To define $\be \co \wh{C} \times \wh{C} \to \ZZ$, it is convenient to work with flat knots. 
A flat Gauss diagram is a trivalent graph with one core circle and (unsigned) directed chords.
For a Gauss diagram $D$, we use $\wbar{D}$ to denote the associated flat Gauss diagram; it is obtained from $D$ by changing the direction of all the negative arrows of $D$ and erasing all the signs. 
We define $\be(x,y)$ following Henrich \cite{Henrich-2010}, who used it to compute the \emph{based matrix}.

\begin{figure}[htb]
\begin{tabular}{cccc}
\begin{tabular}{c} 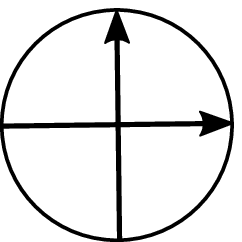 \\ {$\varepsilon=1$} \end{tabular} & 
\begin{tabular}{c} 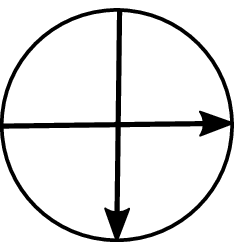 \\ {$\varepsilon=-1$} \end{tabular}  & 
\begin{tabular}{c} 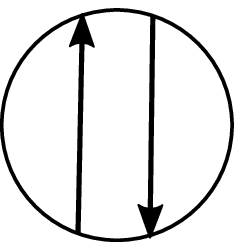 \\ {$\varepsilon=0$}\end{tabular} & \begin{tabular}{c} 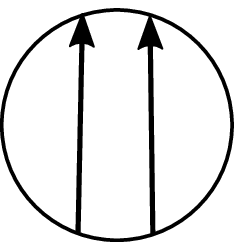 \\ {$\varepsilon=0$} \end{tabular}
\end{tabular}
\caption{Definition of $\varepsilon$.} \label{fig_err}
\end{figure}

If $x\in C$, write $x=(a,b)$ where $a$ is the point on the circle of $\wbar{D}$ meeting the tail of $x$ and $b$ is the point meeting the arrowhead of $x$. Let $(ab)^{\circ}$ denote the interior of the arc in $\wbar{D}$ from $a$ to $b$ (in the counterclockwise direction). Given $x=(a,b)$ and $y=(c,d)$, let $(ab)^{\circ} \to (cd)^{\circ}$ denote the set of arrows beginning in $(ab)^{\circ}$ and ending in $(cd)^{\circ}$. Define:
\[
ab \cdot cd=\#[(ab)^{\circ} \to (cd)^{\circ}]-\#[(cd)^{\circ} \to (ab)^{\circ}].
\]
Here $\#(X)$ denotes the cardinality of $X$. The definition of $\be$ for $x$ and $y$ requires a further error term $\varepsilon$, defined by setting $\varepsilon=-1,0,1$ according to Figure \ref{fig_err}.\footnote{The correction term $\varepsilon$ equals minus the linking number of  $\{a,b\}$ and $\{c,d\}$, viewed as two $S^0$'s in $S^1$, the core circle of $\wbar{D}$.} Define $\be(x,y)=ab\cdot cd+\varepsilon$. If $y=s$, we define $\be(x,s)=\ind(x)=-\be(s,x)$. Lastly set $\be(z,z)=0$ for all $z \in \wh{C}$. With these rules in place, $\be(x,y)=-\be(y,x)$ for all $x,y \in \wh{C}$. Thus, $\be$ is skew-symmetric. 

\bigskip
\noindent
The \emph{graded matrix} is the triple $T=(\wh{C},s,\be)$. Simple graded fillings, graded fillings, and the graded genus $\vartheta(T)$ are defined exactly as in the case of knots in thickened surfaces. The following result relates the two notions of graded matrix. We leave the proof as an exercise.

\begin{proposition} \label{prop_same} Let $\fK$ be a knot in $\Si \times I$ with diagram $\fD$.
Let $D$ be the Gauss diagram associated to $\fD$.  Then the graded matrix of $D$ coincides with the graded matrix of $\fD$.
\end{proposition}

\subsection{Properties of the Graded Genus} Two graded matrices $T_1=(\wh{C}_1,s_1,\be_1)$ and $T_2=(\wh{C}_2,s_2,\be_2)$ are said to be \emph{isomorphic} if there is a bijection $f \co \wh{C}_1 \to \wh{C}_2$ that sends $s_1$ to $s_2$, preserves arrow signs, and satisfies $\be_2(f(x),f(y))=\be_1(x,y)$ for all $x,y \in \wh{C}$. Clearly the isomorphism class of a graded matrix itself is not an invariant of ambient isotopy. However, a graded matrix can always be reduced using a set of rules to a graded matrix whose isomorphism class is an invariant of knots in a fixed thickened surface. The rules (or moves) are as follows:
\begin{enumerate}
\item[($M_1$)] Deletion of an element $x \in C$ such that $\be(x,y)=0$ for all $y \in G$.
\item[($M_2$)] Deletion of an element $x \in C$ such that $\be(x,y)=\be(s,y)$ for all $y \in G$.
\item[($M_3$)] Deletion of a pair $x,y \in C$ of opposite sign such that $\be(x,z)+\be(y,z)=\be(s,z)$ for all $z \in \wh{C}$.
\end{enumerate}
A graded matrix is said to be \emph{primitive} if it does not admit a rule $M_1$, $M_2$, or $M_3$. Reduction to a primitive graded matrix is best done with the matrix $B$ associated $\be$. Rule $M_1$ says to remove row and column $i$ if they are all zeros. Rule $M_2$ says to remove row and column $i$ if they match the first row and column. Rule $M_3$ says to remove rows $i,j$ and columns $i,j$ if they add together to give the first row and column. 

The inverse move of $M_i$ is denoted $M_i^{-1}$. Moves $M_1, M_1^{-1}$ and $M_2, M_2^{-1}$ correspond to $r1$ moves. Move $M_3,M_3^{-1}$ corresponds to the $r2$ move. On the other hand, the $r3$ move has no effect on the isomorphism class of the graded matrix. Turaev proved that if two knots in $\Si \times I$ are equivalent, then they have isomorphic primitive graded matrices. Isomorphic graded matrices necessarily have the same graded genus. In addition, the graded genus of a graded matrix is equal to the graded genus of its reduction to a primitive matrix \cite{Turaev-2008a}. Thus the graded genus is an invariant of knots in a fixed thickened surface. 

Graded matrices $T_1,T_2$ are said to be \emph{concordant}\footnote{Turaev uses the word \emph{cobordant}.} if $\vartheta(T_1,-T_2)=0$. Turaev proved that the graded matrices of concordant knots in thickened surfaces are concordant. Moreover, if two knots in thickened surfaces are concordant, then they have the same graded genus. It is important to note that in contrast to the ambient isotopy relation, we do not need to fix the surface in concordance.

\begin{theorem} If $K_0,K_1$ are concordant virtual knots with graded matrices $T_0,T_1$, respectively, then $\vartheta(T_0)=\vartheta(T_1)$. 
\end{theorem}
\begin{proof} By \cite{Carter-Kamada-Saito}, we may assume that there are knots $\fK_0$, $\fK_1$ in some thickened surfaces representing the virtual knots $K_0,K_1$, respectively, such that $\fK_0$ is concordant to $\fK_1$. Moreover, if two knots in thickened surfaces represent equivalent virtual knots, then they are also concordant knots in thickened surfaces. It follows from Proposition \ref{prop_same} and the remarks in the preceding two paragraphs that the graded genus of $K_i$ is equal to the graded genus of $\fK_i$ for $i=0,1$. Since $\fK_0$ and $\fK_1$ are concordant, they also have the same graded genus.
\end{proof}

As a direct consequence, we deduce that  the graded genus defines a concordance invariant on virtual knots,  denoted $\vartheta(K)$. In fact, by \cite[Lemma 5.2.1]{Turaev-2008a}, we see that the graded genus gives information on the slice genus via the inequality
$\vartheta(K) \leq 2 g_s(K).$

\subsection{An application of $\vartheta(K)$} 
Given a classical knot $K$ with only positive crossings, one can show that the 4-ball genus of $K$ is equal to its Seifert genus. In particular, any such knot is not slice as long as it is nontrivial. We consider the analogous problem for non-classical virtual knots with only positive crossings, and we show how to make the same conclusion by showing that $\vartheta(K)>0.$ We begin by recalling the Cairns-Elton criterion for deciding planarity of a Gauss word $\om$.
  
Let $\kappa$ be an immersed circle on an orientable surface $\Si$ whose multiple points are all transversal double points. Then $\kappa$ can be described as a Gauss word $\om$ \cite{Cairns-Elton-1993}. The \emph{planarity problem} is to determine when a given $\om$ can be represented by a curve $\kappa$ in $\RR^2$. Flat knots are virtual knots considered equivalent up to a finite number of crossing changes. The Gauss code of a flat knot diagram thus only remembers whether intersecting arcs pass from left-to-right or right-to-left while the diagram is traversed. In the virtual knot terminology, the planarity problem then is to determine when the Gauss code of a flat knot diagram is a flat classical knot diagram. In other words, whether or not the diagram can be drawn without virtual crossings.

Recall that the graded matrix of a Gauss diagram $D$ of a virtual knot $K$ is the triple $(\wh{C},s,\be)$, Here $\be \co \wh{C} \times \wh{C} \to \ZZ$ is completely determined by the flat Gauss diagram $\wbar{D}$. In \cite{Cairns-Elton-1993}, a solution to the planarity problem was given that directly translates into the following condition: $\wbar{D}$ is a classical flat diagram if and only if $\be \equiv 0$.

A virtual knot $K$ is said to be \emph{positive} ({\emph{negative}) if it can be represented by a Gauss diagram $D$ with all positive (all negative, respectively) crossing. 

\begin{theorem} \label{gg-app}
Suppose $K$ is a non-classical virtual knot. If $K$ is positive or negative, then the graded genus satisfies $\vartheta(K) \neq 0$. In particular, $K$ is not slice.
\end{theorem}
\begin{proof} Suppose that $D$ is a Gauss diagram of $K$ having all positive or all negative crossings. Let $(\wh{C},s,\be)$ be the graded matrix for $D$. Then the graded matrix has only one simple filling, namely itself.  Since there is only one simple filling, the graded genus is just half the rank of the matrix of $\be$. The only matrix having rank $0$ is the zero matrix. However, since $\wbar{D}$ is non-planar, the Cairns-Elton condition implies that $\be$ is not the zero matrix. Thus $\vartheta(K) \neq 0.$ 
\end{proof}

\begin{corollary} If $K$ is a non-classical slice virtual knot, then every diagram for $K$ must have a non-empty set of positive crossings and a non-empty set of negative crossings.  
\end{corollary}

\subsection*{A graded genus calculator} The odd writhe, Henrich-Turaev polynomial, and writhe polynomial are all easily computable by hand. The graded genus, on the other hand, is more challenging. To aid the reader, a graded genus calculator is available online at \cite{Boden-Chrisman-Gaudreau-2017t}. 
%The calculator is in the Wolfram CDF format. If a \emph{Mathematica} license is locally available, download the file and execute. Otherwise,  the CDF player is available online for free \href{https://www.wolfram.com/cdf-player/}{download}. 

%\begin{figure}[htb]
%\begin{tabular}{c} \\
%\def\svgwidth{4in}\input{Figures/calculator.eps_tex} \end{tabular}
%\caption{Using the graded genus calculator on $4.105$.} \label{fig_calculator}
%\end{figure}

%To compute the graded genus, draw in the Gauss diagram of a virtual knot as follows. First select the total number of arrows needed (up to ten). Then select the number of positive arrows (drawn as $\oplus$). Pushing the ``Go!'' button creates a diagram of a trivial knot. Using the mouse or a finger, drag the endpoints of the arrows to their desired positions. Invariants can then be computed by pushing the appropriate buttons. A screenshot of the calculator in action is given Figure \ref{fig_calculator}.

%%%%%%%%%%%%%%%%%%%%%%%%%%%%%%%%%%%%%%%%%%%%%%%%%%%%%%%%%%%%%%%%%%%%%%%%%%%%%%%%%%%%%%%%%%%%%%%%%%%%

\section{Discussion} \label{sec-discussion}
Clearly, there is much more work to be done, and future progress will depend on developing new concordance invariants in the virtual category. One promising line of research is Khovanov homology and its associated Rasmussen invariant, and that approach has already been fruitful, as evidenced by the results in \cite{Dye-Kaestner-Kauffman-2014} and in the recent preprint of Rushworth \cite{Rushworth-2017}. 

For classical knots, the knot signature was the first and is still arguably the most important concordance invariant. The signature is defined in terms of the Seifert pairing, and since Seifert surfaces do not generally exist for virtual knots, extending the knot signature to the virtual setting is a key problem. In their paper \cite{Im-Lee-Lee-2010}, 
Im, Lee, and Lee use Goeritz matrices to define
signature-type invariants for virtual knots with checkerboard colorings. These invariants depend on the choice of checkerboard coloring, so any given virtual knot will typically have two signatures, one for each choice of black-white surface.

The invariants are relatively straightforward to compute, and they often vanish on slice knots, but there are a few exceptions. For instance, the signatures are non-zero for the  virtual knots 5.2024, 5.2132, 6.73583, and 6.75348. Because each of these four knots is slice, this shows that the signature-type invariants defined in \cite{Im-Lee-Lee-2010} are not actually invariant under virtual knot concordance. In \cite{Boden-Chrisman-Gaudreau-2017a}, which is a companion to this one, we develop an alternative approach to defining signature invariants for almost classical knots. (A virtual knot is \emph{almost classical} if it can be represented as a homologically trivial knot in a thickened surface.) Such knots admit Seifert surfaces, and we use the linking pairing and resulting Seifert matrices to define signatures, twisted signatures, and Alexander-Conway polynomials (cf. \cite{Boden-Gaudreau-Harper-2016}) and derive slice obstructions and slice genus bounds from them.

We close with some interesting questions and open problems.

\begin{enumerate}
\item Ribbon knots can often be sliced in different ways, and ribbon presentations, defined in \cite[\S 2.2]{Carter-Kamada-Saito-2004},  can be \emph{simply equivalent} or \emph{stably equivalent}. Simple equivalence implies stable equivalence but not vice versa; see \cite[Example 2.12]{Carter-Kamada-Saito-2004}. 
What are the corresponding notions for virtual knots? Are any two ribbon presentations of the same virtual knot necessarily stably equivalent?

\item If a classical knot is virtually slice, then it is classically slice (see \cite{Boden-Nagel-2016} for a proof). Does the virtual slice genus of a classical knot always equal its classical slice genus?

\item As slice obstructions, does the graded genus dominate the writhe polynomial? I.e. does $\vartheta(K) =0 \Rightarrow W_K(t)=0$ for all virtual knots $K$?

\item For all known slice virtual knots, we have $H_K(s,t,q)=0.$ Here $H_K(s,t,q)$ denotes the virtual Alexander polynomial, which was introduced in \cite{Boden-Dies-Gaudreau-2015} and  is essentially equal to the Sawollek polynomial $G_K(s,t)$ \cite{Sawollek, Silver-Williams-2003}.  Do these invariants vanish for all slice knots? Are they concordance invariants?  
\end{enumerate}

\begin{conjecture} \label{conj-HK}
If $K$ is a virtual knot and is slice, then  $H_K(s,t,q)=0.$  
\end{conjecture}
Computations for virtual knots with up to six crossings support affirmative answers to all questions (2), (3), and (4) above. For example, if $K$ is a classical knot whose 4-ball genus satisfies $2g_4(K) = |\si(K)|,$ then the results of \cite{Boden-Chrisman-Gaudreau-2017a} show the same is true for the virtual slice genus $g_s(K)$, and it follows that $g_s(K) = g_4(K)$ for such knots.

The slice status of the virtual knots 4.12, 5.93, 5.344, 5.212, 5.919,  5.1034, 5.2351, 5.2430, and 5.2435 is unknown, and  all of them have trivial graded genus but nontrivial virtual Alexander polynomial. If the Conjecture \ref{conj-HK} were true, then it would imply that none of them are slice. The same reasoning would apply to another 182 of the 6-crossing knots whose slice status is unknown. The net effect would be a significant reduction in the number of ``unknown'' cases in Table \ref{table-1}.

\section{Tables and Tabulation}

\subsection{Method of tabulation} For virtual knots with all positive or all negative crossings, the slice genus is determined by Theorem \ref{thm-DKK}. For virtual knots having both positive and negative crossings, Theorem \ref{g-s-bound} gives a useful upper bound which equals 1 for virtual knots with three or four crossings and 2 for virtual knots with five or six crossings. The graded genus also provides a lower bound on the slice genus. Gauss diagram surgeries provide a fast way to identify slice virtual knots. Furthermore, the arrow operations $\{\textbf{cc},\textbf{cd},\textbf{sc},\textbf{or}\}$ and the crossed saddle move all produce genus one cobordisms.  As arrow operations are easy tasks for a computer to perform, this suggests an industrial approach to computing the slice genus for the 92800 distinct virtual knots in Green's table. Indeed, if one of these operations produces a slice knot, then the operation identifies a genus one cobordism to the unknot.  
 
In general it is difficult for a computer to recognize if a virtual knot is slice. Therefore, the first task is to determine all slice knots. As seen in Table \ref{table-1}, the graded genus is a useful slice obstruction. We attempted to slice (by hand) all knots with $\vartheta=0$. Successes for the four and five crossing knots are displayed in Figure \ref{fig-slice}. A list of the 1237 successes for six crossing knots, along with slicings, can be found online at \cite{Boden-Chrisman-Gaudreau-2017t}. 

Given an arbitrary virtual knot, we apply an arrow operation and compare the resulting virtual knot to those on this list of slice virtual knots. If it (or one of its symmetries) is on the list, then  the original knot has slice genus at most one. The next theorem gives a useful sufficient condition for a  virtual knot with four or fewer crossings to be slice.

\begin{theorem}[SliceQ] \label{thm_sliceq} Let $K$ be a virtual knot diagram having four or fewer classical crossings. If the Kauffman bracket polynomial $\langle K \rangle = 1$ and the graded genus $\vartheta(K) =0$, then $K$ is slice.
\end{theorem}
\begin{proof} By assumption, the crossing number of $K$ is at most 4. If $K$ is equivalent to the trivial knot, then $K$ has Kauffman bracket 1 and graded genus zero. It is also slice, so the theorem is true in this case. From Table \ref{table-3}, it follows that virtual knots with crossing number two have non-zero graded genus. Thus we can ignore this case. For crossing number 3, the only knot with trivial graded genus is 3.6, which is the %left-handed classical 
trefoil. Its Kauffman bracket is not 1, so this case can be ignored as well. For the 108 knots having crossing number 4, one can compute the graded genus and the Kauffman bracket. The only knots having having both graded genus 0 and Kauffman bracket 1 are 4.55, 4.56, 4.59, 4.72, 4.76, 4.77, and 4.98. As these knots are slice (see Figure \ref{fig-slice}), the theorem follows. 
\end{proof}

Putting this all together gives the following method of tabulation. Let $K$ be a virtual knot of crossing number $N$. If $K$ has only crossings of the same sign, use Theorem \ref{thm-DKK}. Let $\ell_N$ be the list of $N$ crossing number knots that are known to be slice, together with all their symmetries. The remainder of the algorithm depends on $N$. For virtual knots $K$ with up to four crossings,  Theorem \ref{thm_sliceq} applies to show that if  the following conditions are satisfied, then $K$ is slice: 
\begin{equation} \label{eq-sliceQ}
\langle K \rangle = 1 \quad \text{and} \quad \vartheta(K) =0.
\end{equation}
  
\smallskip  \noindent
{\bf Case $\boldsymbol{N = 4}$:} For each crossing of $K$, perform arrow operations from $\{\textbf{cd},\textbf{cc},\textbf{sc},\textbf{or}\}$ to obtain a knot $K'$. For $\textbf{cd}$, $K'$ is a virtual knot with three or fewer crossings, and it is slice if and only if it is equivalent to the unknot. In this case, equation \eqref{eq-sliceQ} provides a necessary and sufficient condition for sliceness. In the case of $\{\textbf{cc},\textbf{sc},\textbf{or}\}$, $K'$ has four or fewer crossings. If it is in $\ell_4$, then $K$ has a genus one cobordism to the unknot. Otherwise, we check if $K'$ satisfies \eqref{eq-sliceQ}. If so, then $K'$ is slice and it follows that $K$ admits a genus one cobordism to the unknot. 
 
\smallskip \noindent
 {\bf Case  $\boldsymbol{N=5}$:} For each crossing of $K$, perform each of the arrow operations from $\{\textbf{cd},\textbf{cc},\textbf{sc},\textbf{or}\}$ to obtain a knot $K'$. For $\textbf{cd}$, $K'$ is a virtual knot with four or fewer crossings. In this case, it is sufficient to check if $K'$ is in $\ell_4$ or satisfies \eqref{eq-sliceQ}. In the case of $\{\textbf{cc},\textbf{sc},\textbf{or}\}$, $K'$ is a five crossing knot. If $K'$ is in $\ell_5$, then $K$ has a genus one cobordism to the unknot. If not, observe that $\textbf{cc},\textbf{sc},\textbf{or}$ may have introduced a $r2$ move. Removing the $r2$ move creates a knot having $3$ crossings. If the new knot satisfies \eqref{eq-sliceQ}, then $K'$ is slice and $K$ has a genus one cobordism to the unknot.  If all arrow operations fail, apply the crossed saddle move to each pair of intersecting arrows. This creates a diagram $K''$ having $3$ arrows. If $K''$ satisfies \eqref{eq-sliceQ}, then $K$ has a genus one cobordism to the unknot.
 
\smallskip \noindent
 {\bf Case  $\boldsymbol{N=6}$:} Proceed as in the $N=4,5$ cases to obtain a knot $K'$. For $\textbf{cd}$, $K'$ is a five crossing knot. If $K'$ is in $\ell_5$, then we are done. Otherwise, an $r1$ or $r2$ move may have been introduced. If so, remove such arrows and see if the resulting knot, which has at most 4 crossings, satisfies \eqref{eq-sliceQ}. For $\textbf{cc},\textbf{sc},\textbf{or}$ proceed as in the $N=5$ case. If all arrow operations fail, apply the crossed saddle move to pairs of intersecting arrows in the Gauss diagram.

\subsection{The tables} The method of the previous section was implemented in \emph{Mathematica}. The results for virtual knots with four or fewer crossings is given in Table \ref{table-3}, and the full dataset of results for virtual knots with four, five, and six crossings can be found online at \cite{Boden-Chrisman-Gaudreau-2017t}.

The dataset table contains more detailed information about each virtual knot, including its slice status and upper/lower bounds on its slice genus. In case $g_s(K)=1$, it indicates a method for realizing the genus one cobordism to the unknot with a specific virtual unknotting operation $\{ \textbf{cc},\textbf{cd},\textbf{sc},\textbf{or}\}$ applied to a specific chord, or with a crossed saddle move applied to two specific chords.

A snapshot of the six crossing table is given in Figure \ref{fig_table}. The first two columns in this table give the name and Gauss code of the virtual knot. The fourth column ``GG'' gives the graded genus. Half of this value provides a lower bound for the slice genus. The method used for the upper bound is given in the third column. See below for a legend:

\begin{figure}[htb]
\begin{tabular}{c} \\
\def\svgwidth{5in}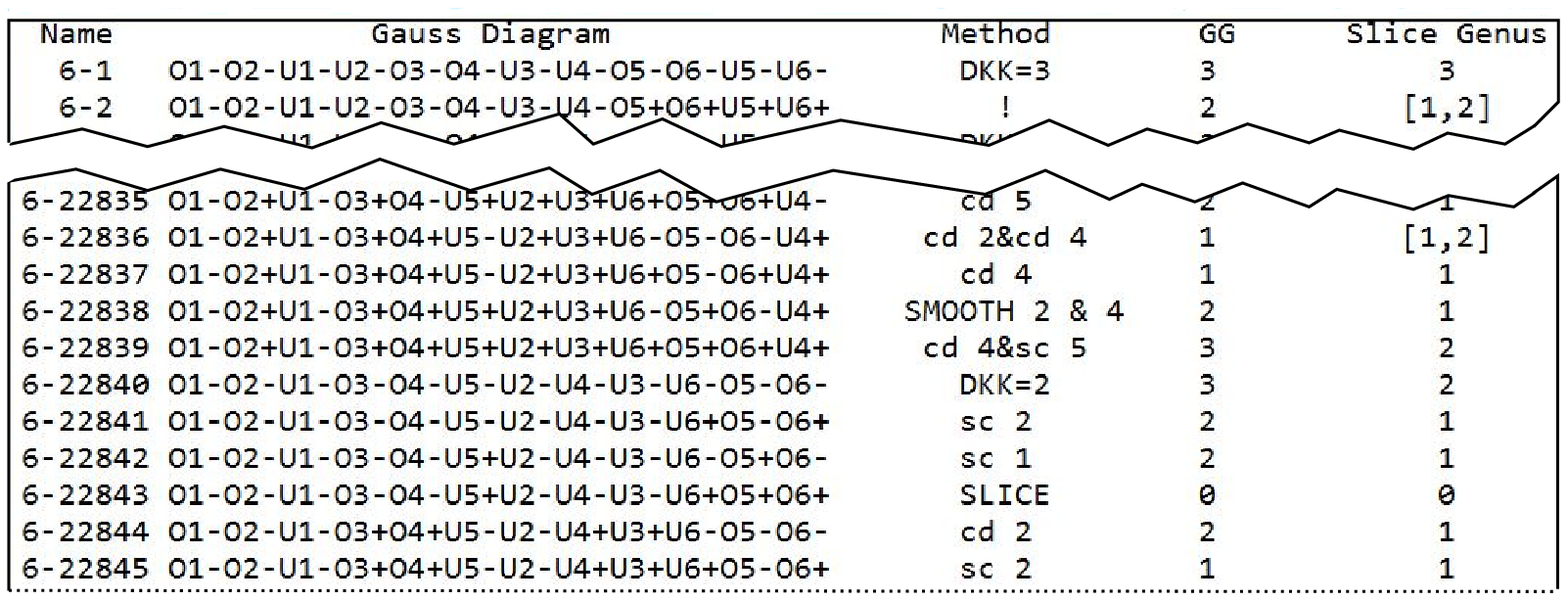 \end{tabular}
\caption{A snapshot of the 6 crossing table.} \label{fig_table}
\end{figure}

\begin{itemize}
\item ``SLICE''--The knot is slice. Explicit slicings are in Figure \ref{fig-slice} for virtual knots with four and five crossings, and in the dataset \cite{Boden-Chrisman-Gaudreau-2017t} for knots with six crossings.

\item ``DKK=$\#$''--Gauss code has only positive or negative crossings and the Dye-Kaestner-Kauffman theorem gives $\#$ as the slice genus. 
\item ``xy $\#$'', xy$\in \{\textbf{cd}, \textbf{or}, \textbf{sc}, \textbf{cc}\}$--The operation xy on the indicated arrow  $\#$ gives a genus one cobordism to the unknot.
\item ``SMOOTH $\#1 \, \& \, \#2$''--A crossed saddle move on arrows $\#1$ and $\#2$ gives a genus one cobordism to the unknot.
\item ``wx $\#1$ $\&$ yz $\#2$'', wx, yz$\in \{\textbf{cd}, \textbf{or}, \textbf{sc}, \textbf{cc}\}$--A genus \underline{two} cobordism to the unknot is obtained by performing the stated operations on the given arrows $\#1$ and $\#2$.
\item ``!''--All above methods failed.
\end{itemize}

Note that for the six crossing table, at most one successful method for the upper bound is given for each knot. More than one method might be successful. Due to the large number of six crossing number knots, the program was written to stop looking as soon as a successful method was found. For the four and five crossing knots, all successful arrow operations are indicated in the third column of the table. If all arrow operations failed, the crossed saddle move was tried. The same notation as above is used to indicate success. Note that the crossed saddle move was never needed to achieve an upper bound of $1$ in the four crossing case. For typographical reasons, applications of the Dye-Kaestner-Kauffman theorem are placed in a separate column for the four and five crossing tables.

\subsection*{Acknowledgements} We would like to thank J. Scott Carter and Andrew Nicas for useful discussions, as well as Louis Kauffman and William Rushworth for their input.  H. Boden was supported by a grant from the Natural Sciences and Engineering Research Council of Canada, M. Chrisman was supported by a Monmouth University Creativity and Research Grant, and
R. Gaudreau was supported by a scholarship from the National Centre of Competence in Research SwissMAP.

\vfill
\newpage

\begin{table}[H]  
\begin{tabular}{cc}
\begin{minipage}{0.32\textwidth}
\begin{tabular}{|c|c|c|}
\hline

Virtual & Graded & Slice \\ 
knot & genus &   genus  \\ 
\hline 
\hline
2.1 & 1 & 1  \\ \hline
3.1 & 1 & 1  \\ \hline
3.2 & 1 & 1  \\ \hline
3.3 & 2 & 1  \\ \hline
3.4 & 1 & 1  \\ \hline
3.5 & 1  & 1 \\ \hline
{\bf 3.6} & 0  & 1 \\ \hline
3.7 & 1 & 1  \\ \hline
4.1 & 2 & 2  \\ \hline
4.2 & 1 & 1  \\ \hline
4.3 & 2 & 2  \\ \hline
4.4 & 1  & 1 \\ \hline
4.5 & 1 & 1  \\ \hline
4.6 & 1 & 1 \\ \hline
4.7 & 2 & 2 \\ \hline
4.8 & 0  & 0 \\ \hline
4.9 & 2 & 2 \\ \hline
4.10 & 2 & 1  \\ \hline
4.11 & 1 & 1  \\ \hline
4.12 & 0 & 0 or 1 \\ \hline
4.13 & 1 & 1   \\ \hline
4.14 & 1 & 1   \\ \hline
4.15 & 2 & 2   \\ \hline
4.16 & 1 & 1   \\ \hline
4.17 & 1 & 1   \\ \hline
4.18 & 2 & 1   \\ \hline
4.19 & 1 & 1   \\ \hline
4.20 & 1 & 1   \\ \hline
4.21 & 1 & 1   \\ \hline
4.22 & 1 & 1   \\ \hline
4.23 & 2 & 1   \\ \hline
4.24 & 1 & 1   \\ \hline
4.25 & 2 & 1   \\ \hline
4.26 & 1 & 1   \\ \hline
4.27 & 1 & 1   \\ \hline
4.28 & 1 & 1   \\ \hline
4.29 & 2 & 2   \\ \hline

4.30 & 1 & 1   \\ \hline
4.31 & 1 & 1   \\ \hline

\end{tabular}
\end{minipage}

\begin{minipage}{0.32\textwidth}

\begin{tabular}{|c|c|c|}
\hline
Virtual & Graded & Slice \\ 
knot & genus &   genus  \\ 
\hline 
\hline
4.32 & 1 & 1   \\ \hline
4.33 & 1 & 1   \\ \hline
4.34 & 1 & 1   \\ \hline
4.35 & 1 & 1   \\ \hline
4.36 & 1 & 1   \\ \hline
4.37 & 1 & 2   \\ \hline
4.38 & 2 & 1   \\ \hline
4.39 & 2 & 1   \\ \hline

4.40 & 1 & 1   \\ \hline
4.41 & 1 & 1   \\ \hline
4.42 & 1 & 1   \\ \hline
4.43 & 1 & 1   \\ \hline
4.44 & 1 & 1   \\ \hline
4.45 & 2 & 1   \\ \hline
4.46 & 1 & 1   \\ \hline
4.47 & 1 & 1   \\ \hline
4.48 & 2 & 2   \\ \hline
4.49 & 2 & 1   \\ \hline

4.50 & 1 & 1   \\ \hline
4.51 & 1 & 1   \\ \hline
4.52 & 1 & 1   \\ \hline
4.53 & 1 & 2   \\ \hline
4.54 & 1 & 1   \\ \hline
4.55 & 0 & 0   \\ \hline
4.56 & 0 & 0   \\ \hline
4.57 & 2 & 1   \\ \hline
4.58 & 0 & 0   \\ \hline
4.59 & 0 & 0   \\ \hline

4.60 & 1 & 1   \\ \hline
4.61 & 2 & 2   \\ \hline
4.62 & 1 & 1   \\ \hline
4.63 & 1 & 1   \\ \hline
4.64 & 1 & 1   \\ \hline
4.65 & 1 & 1   \\ \hline
4.66 & 1 & 1   \\ \hline
4.67 & 1 & 1   \\ \hline
4.68 & 1 & 1   \\ \hline
4.69 & 2 & 2   \\ \hline
4.70 & 1 & 1   \\ \hline

\end{tabular}
\end{minipage}

\begin{minipage}{0.32\textwidth}

\begin{tabular}{|c|c|c|}
\hline
Virtual & Graded & Slice \\ 
knot & genus &   genus  \\ 
\hline 
\hline

4.71 & 0 & 0   \\ \hline
4.72 & 0 & 0   \\ \hline
4.73 & 1 & 2   \\ \hline
4.74 & 1 & 1   \\ \hline
4.75 & 0 & 0   \\ \hline
4.76 & 0 & 0   \\ \hline
4.77 & 0 & 0   \\ \hline
4.78 & 2 & 2   \\ \hline
4.79 & 2 & 1   \\ \hline

4.80 & 2 & 1   \\ \hline
4.81 & 2 & 1   \\ \hline
4.82 & 2 & 1   \\ \hline
4.83 & 1 & 1   \\ \hline
4.84 & 1 & 1   \\ \hline
4.85 & 1 & 1   \\ \hline
4.86 & 1 & 1   \\ \hline
4.87 & 2 & 1   \\ \hline
4.88 & 1 & 1   \\ \hline
4.89 & 2 & 1   \\ \hline

4.90 & 0 & 0   \\ \hline
4.91 & 1 & 2   \\ \hline
4.92 & 1 & 1   \\ \hline
4.93 & 2 & 1   \\ \hline
4.94 & 1 & 1   \\ \hline
4.95 & 1 & 1   \\ \hline
4.96 & 1 & 1   \\ \hline
4.97 & 1 & 1   \\ \hline
4.98 & 0 & 0   \\ \hline
4.99 & 0 & 0   \\ \hline

4.100 & 1 & 2   \\ \hline
4.101 & 1 & 1   \\ \hline
4.102 & 1 & 1   \\ \hline
4.103 & 2 & 1   \\ \hline
4.104 & 1 & 1   \\ \hline
4.105 & 1 & 1   \\ \hline
4.106 & 1 & 1   \\ \hline
4.107 & 1 & 1   \\ \hline
{\bf 4.108} & 0 & 1   \\ \hline
\multicolumn{3}{|c|}{Fin} \\ \hline 
\end{tabular}
\end{minipage}
\end{tabular}
\bigskip
\caption{The graded genus and slice genus for virtual knots up to four crossings. Boldface is used for the classical knots $3_1={\bf 3.6}$ and $4_1={\bf 4.108}.$}
\label{table-3}
\end{table}

\bibliographystyle{halpha}    % a different bib style (requires the file halpha.bst)
\bibliography{bibmaker}

%\printbibliography

\newpage

\begin{figure}[H] {\small
 4.8 \hspace{-0.5cm} \includegraphics[scale=0.80]{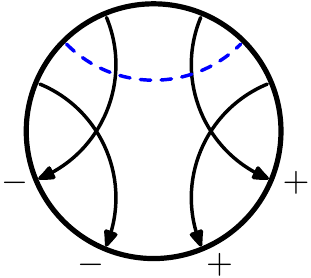} \hspace{0.2cm}
4.55 \hspace{-0.2cm}\includegraphics[scale=0.80]{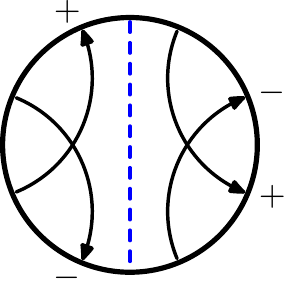} \hspace{0.2cm}  
4.56 \hspace{-0.2cm}\includegraphics[scale=0.80]{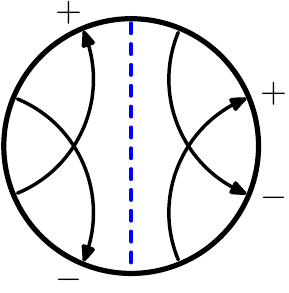}  \hspace{0.2cm}
4.58 \hspace{-0.4cm}\includegraphics[scale=0.80]{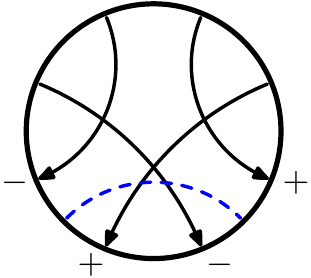} 

\bigskip
4.59 \hspace{-0.4cm}\includegraphics[scale=0.80]{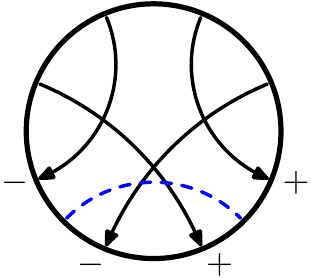}  \hspace{0.2cm}
4.71 \hspace{-0.4cm}\includegraphics[scale=0.80]{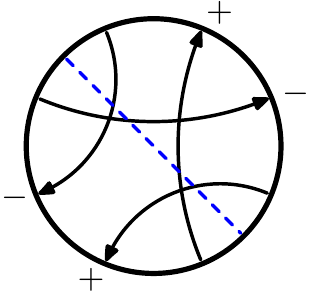}  \hspace{0.2cm}
4.72 \hspace{-0.4cm}\includegraphics[scale=0.80]{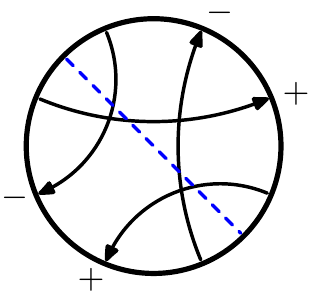} \hspace{0.2cm}
4.75 \hspace{-0.4cm}\includegraphics[scale=0.80]{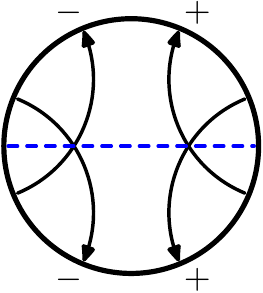} 

\bigskip
4.76 \hspace{-0.2cm} \includegraphics[scale=0.80]{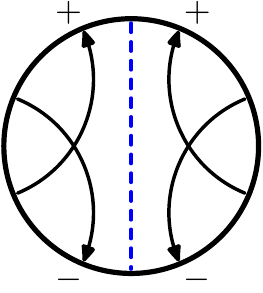} \hspace{0.2cm} 
4.77 \hspace{-0.4cm}\includegraphics[scale=0.80]{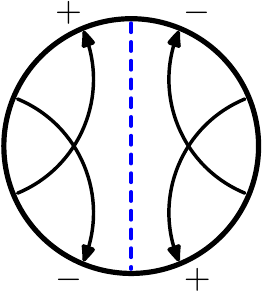} \hspace{0.2cm}  
4.90 \hspace{-0.2cm}\includegraphics[scale=0.80]{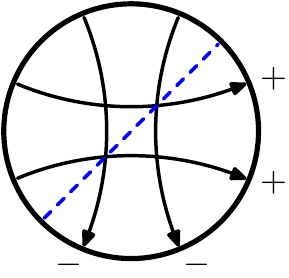} \hspace{0.2cm}
4.98 \hspace{-0.4cm} \includegraphics[scale=0.80]{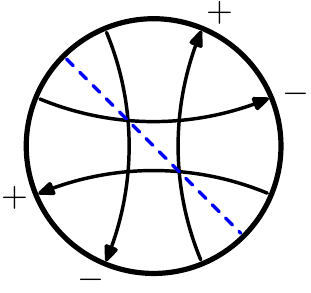} 

\bigskip
4.99 \hspace{-0.2cm}\includegraphics[scale=0.80]{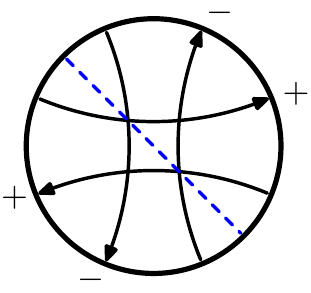} \hspace{0.0cm}
5.274 \hspace{-0.4cm}\includegraphics[scale=0.80]{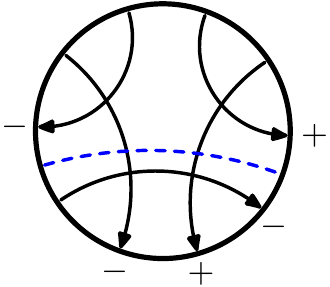} \hspace{0.0cm}
5.278 \hspace{-0.4cm} \includegraphics[scale=0.80]{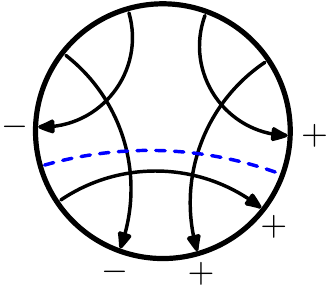} \hspace{0.0cm}
5.280 \hspace{-0.4cm}\includegraphics[scale=0.80]{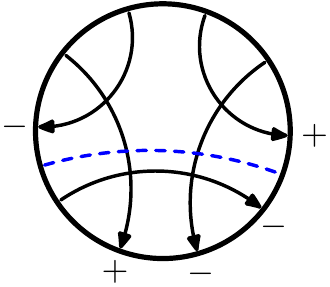} 

\bigskip  
5.284 \hspace{-0.4cm}\includegraphics[scale=0.80]{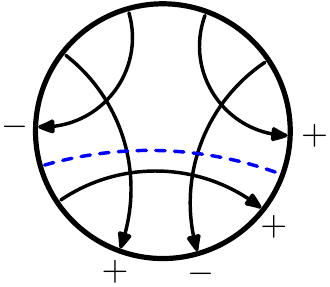} \hspace{0.0cm}
5.592 \hspace{-0.3cm} \includegraphics[scale=0.80]{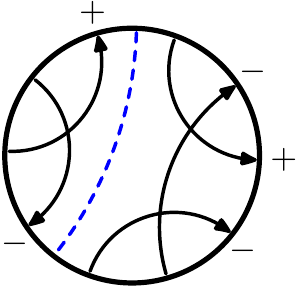} \hspace{0.0cm}
5.594 \hspace{-0.1cm}\includegraphics[scale=0.80]{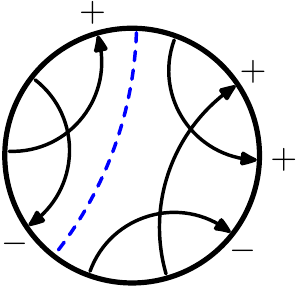} \hspace{0.0cm}  
5.595 \hspace{-0.1cm}\includegraphics[scale=0.80]{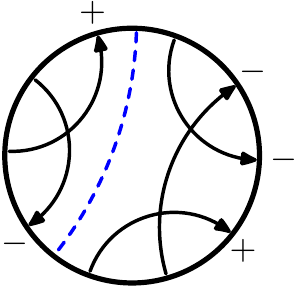} 

\bigskip
5.597 \hspace{-0.1cm} \includegraphics[scale=0.80]{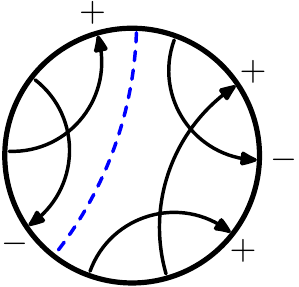} \hspace{0.0cm}
5.812 \hspace{-0.1cm}\includegraphics[scale=0.80]{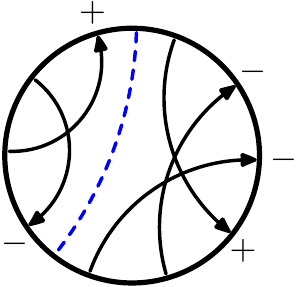} \hspace{0.0cm}
5.813 \hspace{-0.1cm} \includegraphics[scale=0.80]{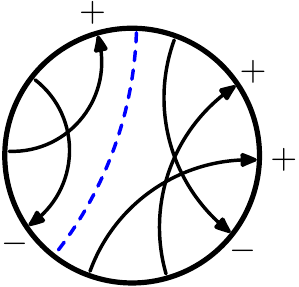}  \hspace{0.0cm}
5.882 \hspace{-0.1cm}\includegraphics[scale=0.80]{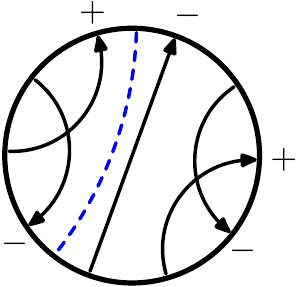}   

\bigskip
5.888 \hspace{-0.3cm}\includegraphics[scale=0.80]{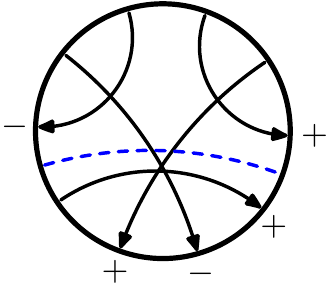} \hspace{0.0cm}  
5.890 \hspace{-0.3cm}\includegraphics[scale=0.80]{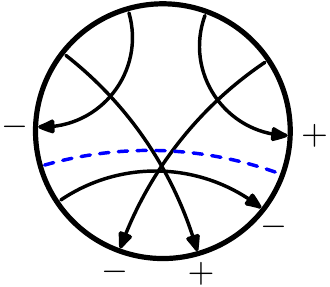} \hspace{0.0cm}
5.893 \hspace{-0.4cm} \includegraphics[scale=0.80]{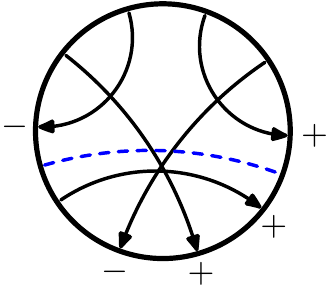} \hspace{0.0cm}
5.1241 \hspace{-0.3cm}\includegraphics[scale=0.80]{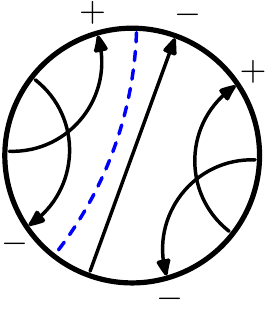}

\bigskip
5.1242 \hspace{-0.3cm} \includegraphics[scale=0.80]{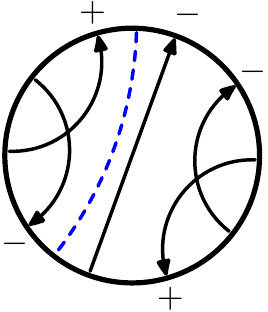}  \hspace{0.0cm}
5.1243 \hspace{-0.2cm}\includegraphics[scale=0.80]{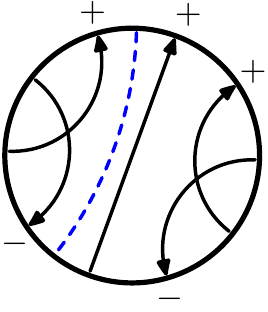}    \hspace{-0.2cm}
5.1347 \hspace{-0.3cm}\includegraphics[scale=0.80]{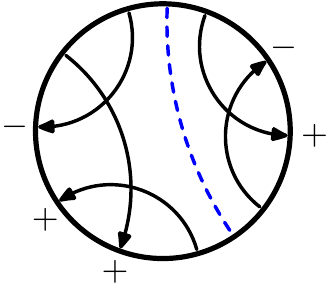} \hspace{-0.2cm}  
5.1348 \hspace{-0.4cm}\includegraphics[scale=0.80]{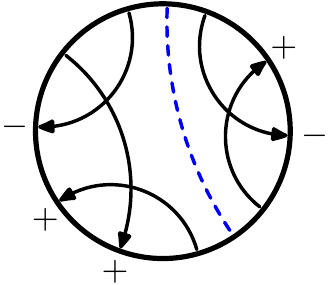} 
}
\end{figure}

\newpage
\begin{figure}[H] {\small

\bigskip

5.1571 \hspace{-0.4cm} \includegraphics[scale=0.80]{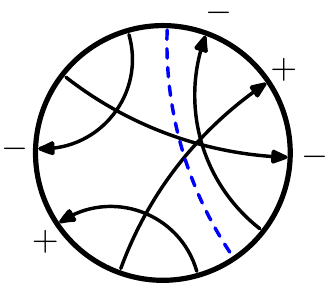} \hspace{-0.2cm}
5.1572 \hspace{-0.3cm}\includegraphics[scale=0.80]{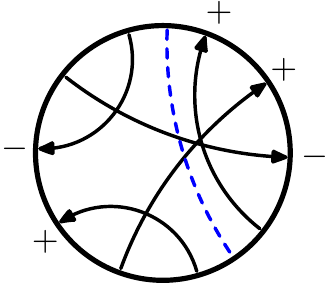} \hspace{-0.2cm}
5.1576 \hspace{-0.3cm} \includegraphics[scale=0.80]{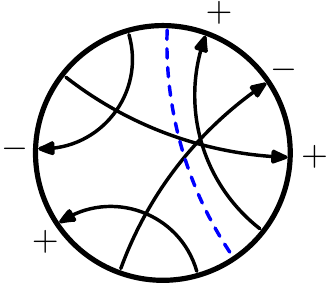}  \hspace{-0.2cm}
5.1585 \hspace{-0.3cm}\includegraphics[scale=0.80]{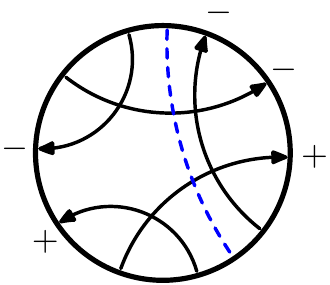}

\bigskip
\small{5.1586} \hspace{-0.3cm}\includegraphics[scale=0.80]{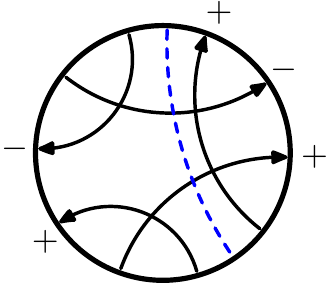} \hspace{-0.2cm}
5.1591 \hspace{-0.3cm}\includegraphics[scale=0.80]{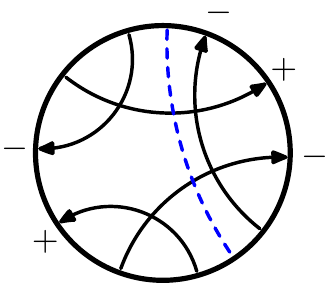} \hspace{-0.2cm}
5.1592 \hspace{-0.3cm} \includegraphics[scale=0.80]{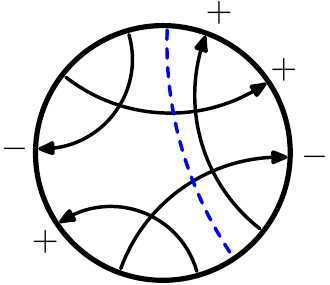} \hspace{-0.2cm}
5.1677 \hspace{-0.3cm}\includegraphics[scale=0.80]{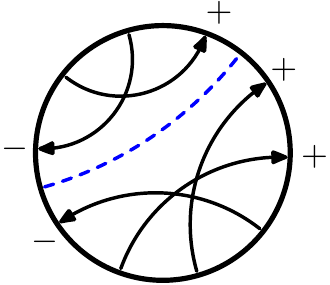} 

\bigskip
5.1678 \hspace{-0.3cm} \includegraphics[scale=0.80]{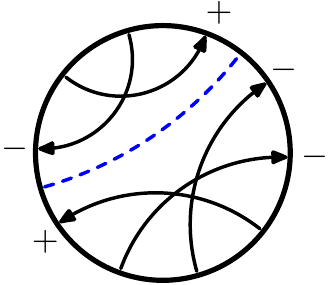} \hspace{0.0cm}
5.1969 \hspace{-0.3cm}\includegraphics[scale=0.80]{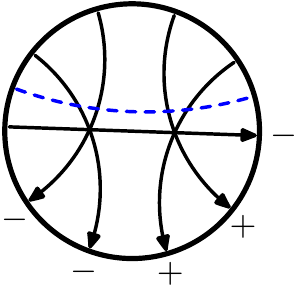}   \hspace{0.0cm}
5.2001 \hspace{-0.3cm}\includegraphics[scale=0.80]{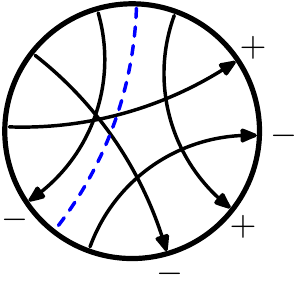} \hspace{0.0cm}  
5.2002 \hspace{-0.3cm}\includegraphics[scale=0.80]{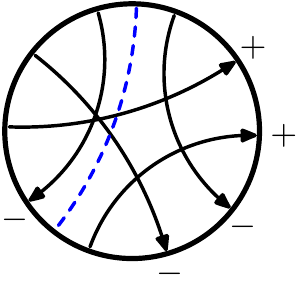}

\bigskip
5.2005 \hspace{-0.3cm} \includegraphics[scale=0.80]{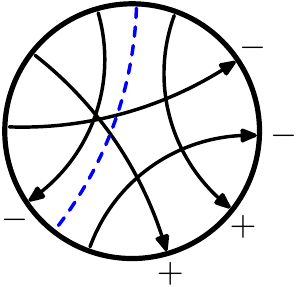} \hspace{0.0cm}
5.2024 \hspace{-0.3cm} \includegraphics[scale=0.80]{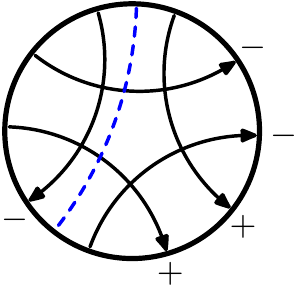}  \hspace{0.0cm}
5.2025 \hspace{-0.3cm}\includegraphics[scale=0.80]{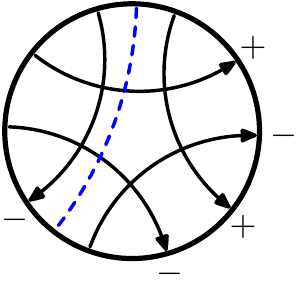} \hspace{0.0cm}  
5.2105 \hspace{-0.3cm}\includegraphics[scale=0.80]{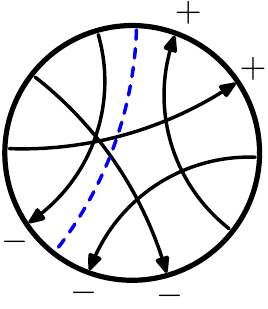}  

\bigskip
5.2106 \hspace{-0.3cm} \includegraphics[scale=0.80]{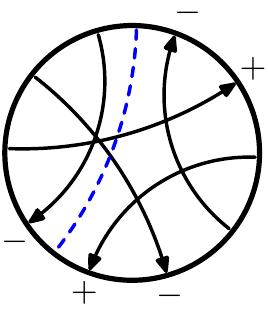} \hspace{0.0cm}
5.2109 \hspace{-0.3cm}\includegraphics[scale=0.80]{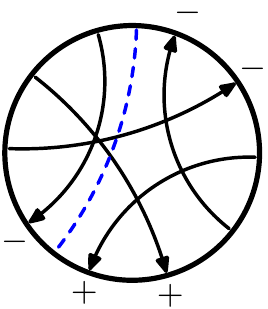} \hspace{0.0cm}
5.2115 \hspace{-0.3cm} \includegraphics[scale=0.80]{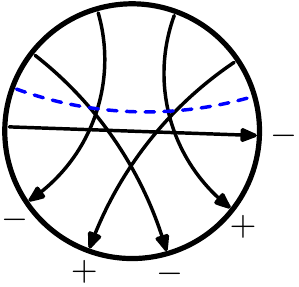}  \hspace{0.0cm}
5.2131 \hspace{-0.3cm}\includegraphics[scale=0.80]{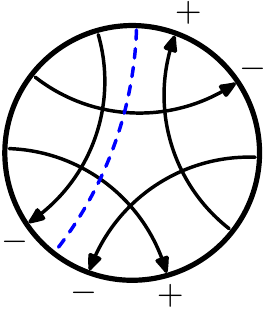}   

\bigskip
\small{5.2132} \hspace{-0.3cm}\includegraphics[scale=0.80]{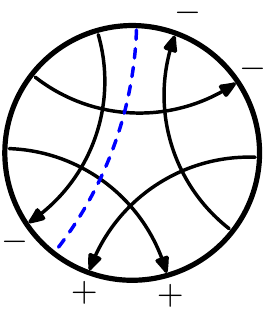} \hspace{0.0cm}  
5.2133 \hspace{-0.3cm}\includegraphics[scale=0.80]{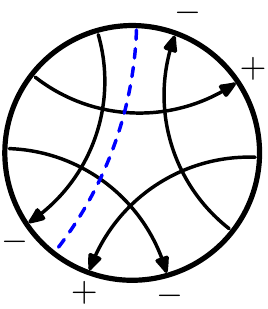} \hspace{0.0cm}
5.2154 \hspace{-0.3cm} \includegraphics[scale=0.80]{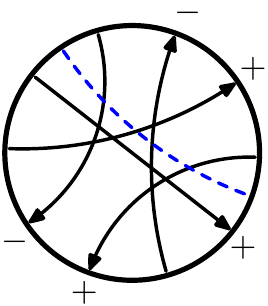} \hspace{0.0cm}
5.2160 \hspace{-0.3cm}\includegraphics[scale=0.80]{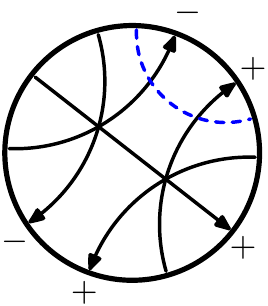}  

\bigskip
5.2212 \hspace{-0.3cm} \includegraphics[scale=0.80]{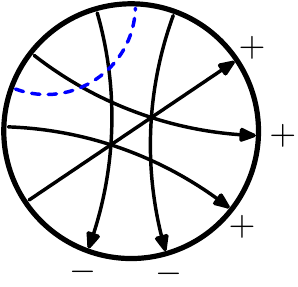} \hspace{0.2cm} 
5.2243 \hspace{-0.3cm}\includegraphics[scale=0.80]{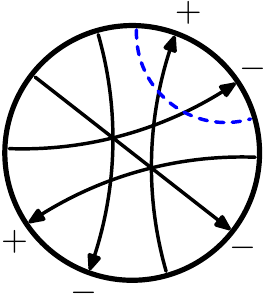}   
}

\bigskip
\caption{Slice Gauss diagrams of virtual knots} \label{fig-slice}
\end{figure}

\end{document}